\documentclass[10pt]{article}

\usepackage[authoryear,round]{natbib}                   
\usepackage{amsmath}                                                        
\usepackage{graphicx}                                                       
\usepackage{subfigure}
\usepackage{enumitem}                                                    
\usepackage{hyperref}
\usepackage{amssymb}                                                        
\usepackage[mathscr]{eucal}                                             
\usepackage{cancel}                                                             
\usepackage[normalem]{ulem}                                                                 
\usepackage{pstricks}
\usepackage{rotating}
\usepackage{lscape}
\usepackage[paperwidth=8.5in,paperheight=11in,top=1.25in, bottom=1.25in, left=1.00in, right=1.00in]{geometry}
\usepackage{mathtools}                                                      
\mathtoolsset{showonlyrefs=true}                                    
\usepackage{fixltx2e,amsmath}                                           
\MakeRobust{\eqref}

\usepackage{mathdots}
\usepackage{amsthm}                                                             
\allowdisplaybreaks                                                             
\theoremstyle{plain}
\newtheorem{theorem}{Theorem}

\newtheorem{lemma}[theorem]{Lemma}                              
\newtheorem{proposition}[theorem]{Proposition}
\newtheorem{corollary}[theorem]{Corollary}
\theoremstyle{definition}
\newtheorem{definition}[theorem]{Definition}
\newtheorem{example}[theorem]{Example}
\newtheorem{remark}[theorem]{Remark}
\newtheorem{assumption}[theorem]{Assumption}

\def \g {{\gamma}}
\def \a {{\alpha}}
\def \b {{\beta}}
\def \xbar {\bar{x}}

\def \ggg {{\nu}}

\newcommand{\<}{\langle}
\renewcommand{\>}{\rangle}
\renewcommand{\(}{\left(}
\renewcommand{\)}{\right)}
\renewcommand{\[}{\left[}
\renewcommand{\]}{\right]}

\newcommand\N{\mathbb{N}}

\newcommand\Eb{\mathbb{E}}

\newcommand\Qb{\mathbb{Q}}
\newcommand\Rb{\mathbb{R}}
\newcommand\Ib{\mathbb{I}}

\newcommand\R{\mathbb{R}}
\newcommand\G{\Gamma}
\newcommand\p{\partial}

\newcommand\Ac{\mathscr{A}}

\newcommand\Bc{\mathscr{B}}

\newcommand\Ec{\mathscr{E}}
\newcommand\Fc{\mathscr{F}}
\newcommand\Gc{\mathscr{G}}

\newcommand\Lc{\mathscr{L}}
\newcommand\Mc{\mathscr{M}}

\newcommand\Vc{\mathscr{V}}

\newcommand\eps{\varepsilon}

\newcommand\Om{\Omega}
\newcommand\sig{\sigma}

\newcommand\gam{\gamma}
\newcommand\Gam{\Gamma}
\newcommand\lam{\lambda}
\newcommand\del{\delta}

\newcommand\xb{\bar{x}}

\newcommand\Hv{\mathbf{H}}
\newcommand\Cv{\mathbf{C}}
\newcommand\mv{\mathbf{m}}

\def \p {{\partial}}

\renewcommand\d{\partial}

\newcommand\ii{\mathtt{i}}
\newcommand\dd{\mathrm{d}}
\newcommand\ee{\mathrm{e}}

\def \phi {{\varphi}}
\def \eps {{\varepsilon}}

\begin{document}

\title{Analytical expansions for parabolic equations}

\author{
Matthew Lorig
\thanks{Department of Applied Mathematics, University of Washington, Seattle, USA.
\textbf{e-mail}: mattlorig@gmail.com.}
\and
Stefano Pagliarani
\thanks{CMAP, Ecole Polytechnique Route de Saclay, 91128 Palaiseau Cedex, France.
\textbf{e-mail}: stepagliara1@gmail.com. Work partially supported by the Chair {\it Financial Risks} of the {\it Risk Foundation}.}
\and
Andrea Pascucci
\thanks{Dipartimento di Matematica, Universit\`a di Bologna, Bologna, Italy.
\textbf{e-mail}: andrea.pascucci@unibo.it}
}

\date{This version: \today}

\maketitle

\begin{abstract}
We consider the Cauchy problem associated with a general parabolic partial differential equation
in $d$ dimensions.  We find a family of closed-form asymptotic approximations for the unique
classical solution of this equation as well as rigorous short-time error estimates.  Using a
boot-strapping technique, we also provide convergence results for arbitrarily large time
intervals.
\end{abstract}


%
%

\section{Introduction}
Asymptotic analysis and perturbation theory have a long history in a variety of fields including
quantum mechanics \cite{sakurai}, classical mechanics \cite{goldstein}, fluid mechanics
\cite{VanDyke1975, Lagerstrom1988, KevorkianCole1996} and mathematical biology \cite{Murray2002}.
More recently, some of techniques from perturbation theory and heat kernel expansions have been
applied to problems arising in mathematical finance: see, for instance, \cite{hagan-woodward,
laborderebook, BenhamouGobetMiri2010, ChengCostanzinoLiechtyMazzucatoNistor2011, fpss}. The
authors of the present manuscript have also made recent contributions in mathematical finance with
a focus on finding closed-form pricing approximations for models both without jumps
\cite{pascucci-parametrix, pascucci} and with jumps \cite{lorig-pagliarani-pascucci-1,
lorig-jacquier-1}, as well as finding closed-form approximations for implied volatility
\cite{lorig-pagliarani-pascucci-2, lorig-pagliarani-pascucci-3, lorig-3}.
\par
In this paper, we shall consider the following Cauchy problem
\begin{equation}
 \begin{cases}
 \left(\d_t + \Ac\right)u(t,x)
    =  0,\qquad  &t\in    [0,T[,\ x\in\mathbb{R}^d, \\
 u(T,x)=  \phi(x), & x \in  \mathbb{R}^d,
\end{cases}\label{eqpde}
\end{equation}
where $\Ac$ is the second order elliptic differential operator with variable coefficients
\begin{align}\label{operator_AA}
 \Ac &= \sum_{i,j=1}^{d}a_{ij}(t,x)\p_{x_{i}x_{j}}+\sum_{i=1}^{d}a_{i}(t,x)\p_{x_{i}}+a(t,x),\qquad t\in\R,\ x\in
\mathbb{R}^d .
\end{align}
Cauchy problems of the form \eqref{eqpde} arise whenever expectations of solutions of stochastic
differential equations are considered.  This is the case, for example, in option pricing.  Cauchy
problems of the form \eqref{eqpde} also arise in quantum mechanics.  However, in this case, one
typically considers initial rather than final data (i.e., $u(0,x) = \varphi(x)$) as well as
imaginary time: $\d_t \to \ii \d_t$.  Indeed, many of the techniques used for finding
approximation solutions of \eqref{eqpde} have been developed by mathematical physicists.
\par
In analyzing \eqref{eqpde}, rather than seek a general solution $u$, one typically seeks the
\emph{fundamental solution} $\Gamma(t,x;T,y)$ (also referred to as the \emph{Green's function}),
which is obtained by setting the final datum equal to a Dirac delta function $\varphi = \del_y$,
and from which the general solution $u$ can be obtained via integration.
\par
Unfortunately, for general $x$-dependent coefficients ($a_{ij}$, $a_i$, $a$), the fundamental
solution is not available in closed-form.  As such, one instead seeks an approximation of the
fundamental solution.  Typically, this is achieved by expressing the operator $\Ac$ as $\Ac =
\Ac_0 + \Bc_1$, where the fundamental solution $\Gam_0$ corresponding to $\Ac_0$ is known in
closed-form and where $\Bc_1 = \Ac - \Ac_0$.  Formally, then, one obtains the fundamental solution
$\Gam$ corresponding to $\Ac$ through a Dyson (also known as Volterra) series expansion
\cite{avramidi2007analytic, berline1992heat}.
\par
While it is a useful tool, the Dyson series has some notable draw-backs.  First, to compute the Dyson series, one must evaluate operator-valued functions of the form
\begin{align}
\Vc(t_0,t_1)
    &:= \exp \( \int_{t_0}^{t_1} \Ac_0(s) \dd s \) \Bc_1(t_1)
            \exp \( \int_{t_1}^{t_0} \Ac_0(s) \dd s \) ,
\end{align}
where we have explicitly indicated the time-dependence in the operators $\Ac_0$ and $\Bc_1$.  It
is rare that the operator $\Vc(t_0,t_1)$ can be computed explicitly and it is certainly not
explicitly computable in the general case.  Second, the Dyson series is typically asymptotically
divergent.  Hence, even if the first few terms of a Dyson series expansion can be computed
explicitly, one is still left to wonder how accurate the truncated series is.
\par
In this paper, rather than expand the operator $\Ac$ as $\Ac = \Ac_0 + \Bc_1$, we expand it as an
infinite sum: $\Ac = \sum_{n \geq 0} \Ac_n$. {The basic ideas of the expansion technique
were introduced in \cite{pagliarani2011analytical}, where $\Ac$ is a differential operator
corresponding to the generator of a scalar diffusion.  These ideas were later extended in
\cite{pascucci} and \cite{lorig-pagliarani-pascucci-1} to the case where $\Ac$ may be an
integro-differential operator corresponding to the generator of a scalar L\'evy-type process.
Both papers mentioned above establish rigorous short-time error bounds for the approximate
fundamental solution of $(\d_t + \Ac)$.  However, the results of these papers are limited to
one-dimension, and leave unanswered some important practical and theoretical issues.  For example:
(i) Is there an explicit (and fully implementable) representation for the approximate solution at
any given order $N$?  (ii) Can the smoothness of the terminal data $\varphi$ be used to establish
a higher order of accuracy of the asymptotic approximation?  (iii) Can anything be said about the
large-time accuracy of the approximation?  We address all of these questions in this manuscript.
In particular, in a multi-dimensional framework we accomplish the following tasks:}
\begin{enumerate}
\item First, we derive fully explicit approximations at any order for fundamental solution $\Gamma(t,x;T,y)$.
We emphasize that, for every $n$, our $n$-th order approximation of the fundamental solution
$\Gam$ is explicit; no integrals or special functions are required.  This is not the case for the
formal Dyson series expansion.
\item {Second, we show how regularity of the terminal datum $\varphi$ can be used to establish a higher order of accuracy for small times.}
\item Third, we prove convergence results on arbitrarily large time intervals.
\end{enumerate}

{On an applied level, the results proved in this manuscript serve as the foundation for some
recent developments in mathematical finance.  More specifically, in
\cite{lorig-pagliarani-pascucci-2}, the authors use the small-time error bounds established here
for solutions $u$ of \eqref{eqpde} in order to prove small-time error bounds for the \emph{implied
volatility} of European Call options in a general \emph{multifactor} local-stochastic volatility
model.  We note that proving the accuracy result for implied volatility depends on exploiting the
smoothness of the terminal datum $\varphi$.}


Our proofs in this manuscript are based on a combination of symmetry properties of Gaussian kernels and (very
general) classical results such as Duhamel's principle, the Chapman-Kolmogorov identity and some
upper bounds for the fundamental solution of the operator $(\d_{t} + \Ac)$. Due to the generality
of the main ingredients in the proofs, our approach opens the door to more general expansions,
which may not necessarily be based on Gaussian kernels.
\par
The analytical techniques presented in this paper were originally developed with applications to
financial mathematics in mind.  However, because we provide a systematic treatment of Cauchy
problem \eqref{eqpde}, including complete and rigorous proofs of error bounds and convergence, we
believe that our results are of interest in other fields in which parabolic equations arise, such
as mathematical biology, chemistry, physics, engineering and economics.
\par
The rest of this paper proceeds as follows: in Section \ref{sec:approximating} we introduce the
idea of expanding the coefficients of $\Ac$ as a sum of polynomial basis functions.  We provide
examples of useful basis functions and list our main assumptions.  Next, in Section
\ref{sec:approximating1}, we present our main results.  Theorem \ref{th:un_general_repres}
provides a closed-form expression for the $n$-th term of the asymptotic expansion of $u$, the
solution of \eqref{eqpde}.  The {theorem} is written in a very general fashion, which allows for
not just a single asymptotic expansion of $u$, but for {\it an entire family of asymptotic
expansions} for $u$.  In Theorem \ref{th:error_estimates_taylor}, we provide small-time error
bounds for our asymptotic approximation of $u$.  And in Theorem
\ref{th:bootstrapping_convergence}, we provide convergence results, which are valid on any finite
time interval.  Next, in Section \ref{sec:model}, we illustrate how the solution to Cauchy problem
\eqref{eqpde} relates to the pricing of derivatives in financial mathematics.  Finally, Sections
\ref{appendix:proof}, \ref{smalla} and \ref{boota} contain the proofs of Theorems
\ref{th:un_general_repres}, \ref{th:error_estimates_taylor} and \ref{th:bootstrapping_convergence}
respectively.

%
%

\section{General expansion basis}\label{sec:approximating}\setcounter{equation}{0}
To begin, we will establish some notation and state our main assumptions. For any $n\in\N_{0}$, we
denote by $C_{b}^{n,1}(\R^{d})$ the class of bounded functions with (globally) Lipschitz
continuous derivatives of order less than or equal to $n$, and by $\left\|f\right\|_{C_{b}^{n,1}}$
the sum of the $L^{\infty}$-norms of the derivatives of $f$ up to order $n$ and the Lipschitz
constants of the derivatives of order $n$ of $f$.
We also denote by $C_{b}^{-1,1}=L^{\infty}$
the class of bounded and measurable functions and set $\left\|\cdot\right\|_{C_{b}^{-1,1}}=
\left\|\cdot\right\|_{L^{\infty}}$.
\par
Throughout the rest of the paper we shall assume that $\overline{T}>0$ and $N\in\N_{0}$ are fixed
and the coefficients of the operator $\Ac$ in \eqref{operator_AA} satisfy the following
assumption.
\begin{assumption}\label{assumption:parabol_holder_bonded}
There exists a positive constant $M$ such that:
\begin{enumerate}
\item[i)] {\it Uniform ellipticity:}
\begin{align}\label{cond:parabolicity}
 M^{-1}|\xi|^2\leq \sum_{i,j=1}^{d}a_{ij}(t,x)\xi_{i}\xi_{j}\leq M |\xi|^2,\qquad
t\in\left[0,\overline{T}\right],\ x,\xi\in\mathbb{R}^d.
\end{align}
\item[ii)] {\it Regularity and boundedness:} \  the coefficients $a_{ij},a_{i},a\in C\left(\left[0,\overline{T}\right]\times\R^{d}\right)$ and for any
$t\in\left[0,\overline{T}\right]$ we have $a_{ij}(t,\cdot),a_{i}(t,\cdot),a(t,\cdot)\in
C^{N,1}_{b}(\R^{d})$, with their $\left\|\cdot\right\|_{C^{N,1}_{b}}$-norms bounded by $M$.
\end{enumerate}
\end{assumption}
Under Assumption \ref{assumption:parabol_holder_bonded} it is well known that, for any
$T\in\left]0,\overline{T}\right]$ and $\phi \in C^{-1,1}_{b}$, the backward parabolic Cauchy
problem \eqref{eqpde} admits a classical solution $u$.  However, in general, the function $u$ is
not known in closed-form and, for practical purposes, must be computed numerically.
\par
In what follows, it will be convenient to rewrite the differential operator \eqref{operator_AA} in the more compact form
\begin{align}\label{operator_A}
 \Ac:= \sum_{|\alpha |\leq 2} a_{\alpha}(t,x) D^{\alpha}_{x},\qquad t\in\R,\ x\in
\mathbb{R}^d,
\end{align}
where by standard notations
\begin{align}
 \alpha=(\alpha_1,\cdots,\alpha_d)\in \mathbb{N}^{d}_{0}, \quad
 |\alpha|=\sum_{i=1}^{d}\alpha_i,\quad D_{x}^{\alpha}=D^{\alpha}=\partial^{\alpha_1}_{x_1}\cdots
\partial^{\alpha_d}_{x_d}.
\end{align}
Below, we will introduce a family of expansion schemes for the operator $\Ac$. Each of the
different families of expansion schemes is based on a different expansion of the coefficients
$(a_\alpha)_{|\alpha|\leq 2}$, and will result in a different approximation for the solution $u$
of \eqref{eqpde} as well as a different approximation for the fundamental solution $\Gam$. Thus,
for any $\alpha\in\mathbb{N}^d_{0}$ with $|\alpha| \le2$, we fix an approximation sequence
$\left(a_{\alpha,n}\right)_{n\ge 0}$ of {continuous} functions
\begin{align}
 a_{\alpha,n}:\left[0,\overline{T}\right]\times\mathbb{R}^d\to\mathbb{R}.
\end{align}
More precisely we introduce the following definition:
\begin{definition}
We say that $\left(a_{\alpha,n}\right)_{0\leq n\leq N}$ is an {\it $N$-th order polynomial expansion} 
if
, for any $t\in\left[0,\overline{T}\right]$, the functions $a_{\alpha,n}(t,\cdot)$ are polynomials with
$a_{\alpha,0}(t,\cdot)=a_{\alpha,0}(t)$.
\end{definition}
The idea behind our approximation method is to choose a polynomial expansion such that the
sequences of partial sums $\sum_{n=0}^N a_{\alpha,n}(t)$ approximate the coefficients
$a_{\alpha}(t,z)$, either pointwise or in some norm. We conclude this section by presenting some practical examples of polynomial expansions.

\begin{example}\label{example:Taylor}(Taylor polynomial expansion)\\
Let Assumption \ref{assumption:parabol_holder_bonded} ii) hold true. 
Then, for any fixed
$\bar{x}\in\R^{d}$, we define $a_{\alpha,n}$ as the $n$-th order term of the Taylor
expansion of $a_{\alpha}$ in the spatial variables around $\xbar$.  That is, we set
\begin{equation}
 a_{\alpha,n}(\cdot,x)
    =\sum_{|\b|=n}\frac{D^{\b}a_{\alpha}(\cdot,\bar{x})}{\b!}(x-\bar{x})^{\b}, \qquad
        { 0\leq n \leq N},\qquad
|\a|
        \le 2,
\end{equation}
where as usual {$\b!=\b_{1}!\cdots\b_{d}!$ and $x^\beta = x_1^{\beta_1} \cdots x_d^{\beta_d}$}.
The expansion proposed in \cite{lorig-pagliarani-pascucci-2} and \cite{lorig-pagliarani-pascucci-3} is the particular case where $d=2$.
\end{example}

\begin{example}\label{example:enhanced_Taylor}(Enhanced Taylor expansion)\\
In the previous example, the $n$-th order term $\Ac_n$ of the polynomial expansion of $\Ac$
coincides with the $n$-order term of the Taylor expansion. More generally, we may define the
$n$-th order term $\Ac_n$ of the polynomial expansion of $\Ac$ so that it coincides with a higher
order Taylor expansion.  Specifically, { assume $N\geq 1$}, and let {$M_{0}=0$ and $(M_n)_{1\leq n\leq N}$ be a non-decreasing
sequence of natural numbers where, in general, $M_n$ may be greater than $n$}.  We may assume that
\begin{equation}
a_{\alpha,0}(\cdot)
    =  a_{\alpha}(\cdot,\bar{x}), \qquad
a_{\alpha,n}(\cdot,x)
  =  \sum_{|\b|=1+M_{n-1}}^{M_n}
                \frac{D^{\b}a_{\alpha}(\cdot,\bar{x})}{\b!}(x-\bar{x})^{\b}, \qquad
{ 1\leq n
    \leq N}, \qquad
|\a|
    \le 2.
\end{equation}
The enhanced Taylor expansion is motivated by the fact that, in the limit as $M_1 \to \infty$ we have that $\Ac_1 = \Ac - \Ac_0 = \Bc_1$.
Thus, in this limit our expansion for $u$ (given in Theorem \ref{th:un_general_repres}) provides an explicit asymptotic representation for the Dyson series expansion.
\end{example}

\begin{example}\label{example:TimeTaylor}(Time-dependent Taylor polynomial expansion)\\
For any fixed $\bar{x}:\Rb_+ \to \R^{d}$, we define $a_{\alpha,n}$ as the $n$-th order term of the
Taylor expansion of $a_{\alpha}$ in the spatial variables around $\xbar$.  That is, we set
\begin{equation}
 a_{\alpha,n}(\cdot,x)
    =\sum_{|\b|=n}\frac{D^{\b}a_{\alpha}(\cdot,\bar{x}(\cdot))}{\b!}(x-\bar{x}(\cdot))^{\b},\qquad { 0\leq n\leq N},\qquad |\a|\le 2 .
\end{equation}
This expansion for the coefficients allows the expansion point $\bar{x}$ of the Taylor series
to evolve in time.  By construction $\Ac_0$ is guaranteed to be the generator of a diffusion
$X^0$.  It is natural, then, to choose $\xb(t)$ to be $\xb(t) =  E\left[X_t^0\right]$, the
expected value of $X_t^0$.
 In \cite{lorig-pagliarani-pascucci-2} this choice results in a highly accurate approximation for
option prices and implied volatility in the \cite{heston1993} model.
\end{example}

\begin{example}\label{example:Hilbert}(Hermite polynomial expansion)\\
Hermite expansions can be useful when the diffusion coefficients are not smooth.  A remarkable
example in financial mathematics is given by the Dupire's local volatility formula for models with
jumps (see \cite{frizyor2013}). In some cases, e.g., the well-known Variance-Gamma model, the
fundamental solution (i.e., the transition density of the underlying stochastic model) has
singularities.  In such cases, it is natural to approximate it in some $L^{p}$ norm rather than in
the pointwise sense. For the Hermite expansion centered at $\xb$, one sets
\begin{equation}
a_{\alpha,n}(t,x)
    =  \sum_{|\b|=n} \< \Hv_\b(\cdot-\bar{x}) , a_{\alpha}(t,\cdot) \>_{\Gam} \Hv_\b(x-\bar{x}),\qquad { 0\leq n\leq N},\qquad |\a|\le 2 ,
\end{equation}
where the inner product $\<\cdot,\cdot\>_\Gam$ is an integral over $\Rb^d$ with a Gaussian
weighting centered at $\xb$ and the functions $\Hv_\beta(x) = H_{\beta_1}(x_1) \cdots
H_{\beta_d}(x_d)$ where $H_n$ is the $n$-th one-dimensional Hermite polynomial (properly
normalized so that $\< \Hv_\alpha, \Hv_\beta \>_\Gam = \delta_{\alpha,\beta}$ with
$\delta_{\alpha,\beta}$ being the Kronecker's delta function).
\end{example}

\section{Main results: closed-form solutions, local and global error bounds}\label{sec:approximating1}\setcounter{equation}{0}
The main idea behind the construction of an approximation for the solution $u$ of \eqref{eqpde} is very intuitive. { We begin this section by presenting the derivation of a formal expansion of $u$. Let us consider a polynomial expansion $\left(a_{\alpha,n}\right)_{n\in \N_0}$ and let us assume that 
the operator $\Ac$ in \eqref{operator_A} can be formally written as}
\begin{align}
  \Ac &=  \sum_{n=0}^\infty \Ac_n, &
  \Ac_n&:= \sum_{|\alpha |\leq 2}  a_{\alpha,n}(t,x) D^{\alpha}_{x}. \label{eq:A.expand}
\end{align}
We now follow the classical approach and expand the solution $u$ of \eqref{eqpde} as follows
\begin{align}
  u  &=  \sum_{n=0}^\infty u_n . \label{eq:v.expand}
\end{align}
Inserting \eqref{eq:A.expand} and \eqref{eq:v.expand} into \eqref{eqpde} we find that the
functions $(u_{n})_{n\geq 0}$ satisfy the following sequence of nested Cauchy problems
\begin{align}\label{eq:v.0.pide}
&\begin{cases}
 (\d_t + \Ac_0 ) u_0(t,x) =  0,\qquad & t\in[0,T[,\ x\in\mathbb{R}^d, \\
 u_0(T,x) =  \phi(x),& x \in\mathbb{R}^d,
\end{cases}
\intertext{and} \label{eq:v.n.pide} &\begin{cases}
 (\d_t + \Ac_0 ) u_n(t,x) =  - \sum\limits_{h=1}^{n} \Ac_h u_{n-h}(t,x),\qquad &  t\in[0,T[,\ x\in\mathbb{R}^d, \\
 u_n(T,x) =  0, &    x \in\mathbb{R}^d.
\end{cases}
\end{align}
Since, by assumption, the functions {$a_{\alpha,0}$} depend only on $t$, the operator $\Ac_{0}$ is
elliptic with time-dependent coefficients. It will be useful to write the operator $\Ac_0$ in the
following form:
\begin{align}\label{def_A0_bis}
\Ac_0
    &=\frac{1}{2} \sum_{i,j=1}^{d} C_{ij}(t)\p_{x_{i}x_{j}} + \langle m(t), \nabla_{x} \rangle
 +\g(t) , &
\langle m(t), \nabla_{x} \rangle
    &=  { \sum_{i=1}^d m_i(t) \d_{x_i} } .
\end{align}
Here the $d\times d$-matrix $C$ is positive definite, uniformly for $t\in[0,T]$, and $m$ and
$\gamma$ are a $d$-dimensional vector and a scalar functions respectively.
\begin{example}
If $d=2$ we have
\begin{align}
 C&=\left(
 \begin{array}{cc}
  2 a_{(2,0),0} & a_{(1,1),0} \\
  a_{(1,1),0}  & 2 a_{(0,2),0}
 \end{array}
  \right),&
 m&= \left(  \hspace{-5pt}
\begin{array}{rl}
 a_{(1,0),0} , & \hspace{-5pt}  a_{(0,1),0}
\end{array}
\hspace{-5pt}  \right), & \gamma &= a_{(0,0),0}.
\end{align}
\end{example}
It is clear that the leading term $u_0$ in the expansion \eqref{eq:v.expand} is explicitly
given by
\begin{align}\label{e10}
 u_0(t,x) &=  \ee^{\int_t^T \g(s) \dd s} \int\limits_{\mathbb{R}^d} \Gamma_{0}(t,x;T,y) \phi(y)\, \dd y ,\quad t<T,\ x\in\mathbb{R}^d,
\end{align}
where $\Gamma_{0}$ is the $d$-dimensional Gaussian density
\begin{align}
 \Gamma_{0}(t,x;T,y) &=  \frac{1}{  \sqrt{(2\pi)^{d}|\Cv(t,T)|} }
    \exp\left(-\frac{1}{2}\langle\Cv^{-1}(t,T) (y - x-\mv(t,T)),
    (y -x-\mv(t,T))\rangle\right) \label{e22and}
\end{align}
with covariance matrix $\Cv(t,T)$ and mean vector $x+\mv(t,T)$ given by:
\begin{align}\label{eq:covariance_mean}
 \Cv(t,T)=  \int_t^T C(s) \dd s,\qquad\qquad \mv(t,T)=   \int_t^T m(s) \dd s.
\end{align}
The first main result of the paper is Theorem \ref{th:un_general_repres} below.  The {theorem}
provides \emph{an explicit representation for each $u_n$} in \eqref{eq:v.expand}.  Remarkably,
every $u_n$ can be written as a finite sum of spatial derivatives acting on $u_0$.
\begin{theorem}\label{th:un_general_repres}
For any $n\geq 1$, the $n$-th term $u_n$ in \eqref{eq:v.expand} is given by
\begin{align}\label{eq:un}
u_n(t,x)
    &=  \Lc^{x}_n(t,T) u_0(t,x), &
t
    &<  T,\, x\in\mathbb{R}^d.
\end{align}
In \eqref{eq:un}, $\Lc_n^{x}(t,T)$ denotes the differential operator acting on the
$x$-variable and defined as
\begin{align}\label{eq:def_Ln}
 \Lc_n^{x}(t,T):=  \sum_{h=1}^n \int_{s_0}^T \dd s_1 \int_{s_1}^T \dd s_2 \cdots \int_{s_{h-1}}^T
\dd s_h
      \sum_{i\in I_{n,h}}\Gc^{x}_{i_{1}}(s_0,s_1) \cdots \Gc^{x}_{i_{h}}(s_0,s_h) ,
\end{align}
where\footnote{ For instance, for $n=3$ we have $I_{3,3}=\{(1,1,1)\}$, $I_{3,2}=\{(1,2),(2,1)\}$
and $I_{3,1}=\{(3)\}$. }
\begin{align}\label{eq:def_Ln_bis}
I_{n,h}
    &=  \{i=(i_{1},\dots,i_{h})\in\mathbb{N}^{h} \mid i_{1}+\dots+i_{h}=n\} , &
1
    & \le h \le n ,
\end{align}
and the operator $\Gc^{x}_{n}(t,s)$ is defined as
\begin{align}\label{def_Gn}
\Gc^{x}_{n}(t,s)
    &:= \sum_{|\alpha| \leq 2} a_{\alpha,n}\big(s,\Mc^x(t,s)\big) D^{\alpha}_{x},
\end{align}
with
\begin{align} \label{eq:M}
\Mc^{x}(t,s) &= x+\mv(t,s)+ \Cv(t,s)\nabla_{x}.
\end{align}
\end{theorem}
\noindent Theorem \ref{th:un_general_repres} will be proved in Section \ref{appendix:proof}.
\begin{remark}
Particular cases of Theorem \ref{th:un_general_repres} have been already stated, devoid of proof,
in \cite{lorig-pagliarani-pascucci-2} and \cite{lorig-pagliarani-pascucci-3}.  In
\cite{lorig-pagliarani-pascucci-2}, only time-homogeneous two-dimensional diffusions are treated.
In \cite{lorig-pagliarani-pascucci-3}, only the Taylor series expansion of $\Ac$ is treated.
\end{remark}
Our second main result consists in local-in-time error bounds for the { $N$-th order} Taylor expansion of Example
\ref{example:Taylor}.  In what follows, it will be helpful to indicate explicitly the dependence on
$\xb$, the expansion point of the Taylor series.  As such, we introduce the following notation:
for $n\le N$ and $\xbar\in\R^{d}$, we set
\begin{align} \label{eq:A.expand.xbar}
 \Ac^{(\bar{x})}_n &= \sum_{|\alpha |\leq 2} a_{\alpha,n}^{(\bar{x})} D_{x}^{\alpha},&
 a_{\alpha,n}^{(\bar{x})}(t,x)&=  \sum_{|\b|=n} \frac{D^{\b}a_{\alpha}(t,\bar{x})}{\b!}(x-\bar{x})^{\b} .
\end{align}
The approximating terms $u_n=u^{(\bar{x})}_n$ in the expansion \eqref{eq:v.expand} solve
\begin{align}\label{eq:v.0.pide.xbar}
  \begin{cases}
 \left(\d_t + \Ac^{(\bar{x})}_0\right) u^{(\bar{x})}_0(t,x) =  0,\qquad &  t<T,\ x\in\mathbb{R}^d, \\
  u^{(\bar{x})}_0(T,x) =  \phi(x), & x \in\mathbb{R}^d,
  \end{cases}
\end{align}
and for { $1\leq n\leq N$}
\begin{align}\label{eq:v.n.pide.xbar}
  \begin{cases}
 \left(\d_t + \Ac^{(\bar{x})}_0\right) u^{(\bar{x})}_n(t,x) =  - \sum\limits_{h=1}^{n} \Ac^{(\bar{x})}_h u^{(\bar{x})}_{n-h}(t,x),\qquad &  t<T,\ x\in\mathbb{R}^d, \\
  u^{(\bar{x})}_n(T,x) = 0, & x \in\mathbb{R}^d.
  \end{cases}
\end{align}
Next, we define the approximate solution at order $N$ for the Taylor expansion centered at $\xbar$ as
\begin{align}\label{Taylor_approx_k_N_xbar}
 \bar{u}^{(\bar{x})}_N(t,x):=
 \sum_{n=0}^N u^{(\bar{x})}_n(t,x).
\end{align}
For the particular choice $\xbar=x$, we simply set
\begin{align}\label{eq:def_ubar_N}
\bar{u}_N(t,x):=\bar{u}^{(x)}_N(t,x) .
\end{align}
We call $\bar{u}_N$ the \emph{$N$-th order Taylor approximation of $u$}.  Analogously, for the fundamental solution $\Gamma$ of $(\p_{t}+\Ac)$, we set
\begin{align}
\bar{\Gamma}_{N}(t,x;T,y)=\bar{\Gamma}^{(x)}_{N}(t,x;T,y). \label{eq:new}
\end{align}

\begin{theorem}\label{th:error_estimates_taylor}
Let Assumption \ref{assumption:parabol_holder_bonded} hold and let $0<T\leq \overline{T}$.
Assume also the initial datum $\phi\in C_{b}^{k-1,1}\left(\mathbb{R}^d\right)$ for some
$0 \leq k\leq 2$. Then we have
\begin{align}\label{eq:error_estimate}
 \left| u(t,x)-\bar{u}_N(t,x) \right| \leq C (T-t)^{\frac{N+k+1}{2}}, \qquad 0\leq t<T,\ x\in\mathbb{R}^d,
\end{align}
where the constant $C$ only depends on $M,N,\overline{T}$ and $\left\|\varphi\right\|_{C^{k-1,1}_{b}}$.  Moreover, for any $\eps>0$ we have
\begin{align}\label{th:error_estim_fund_solution}
 \left| \Gamma(t,x;T,y)-\bar{\Gamma}_{N}(t,x;T,y) \right| \leq C (T-t)^{\frac{N+1}{2}}\Gamma^{M+\varepsilon} (t,x;T,y), \qquad 0\leq t<T,\
 x,y\in\mathbb{R}^d,
\end{align}
where $\Gamma^{M+\varepsilon}(t,x;T,y)$ is the fundamental solution of the $d$-dimensional heat
operator
\begin{align}\label{eq:heatoperator}
H^{M+\varepsilon}=(M+\varepsilon) \sum_{i=1}^d \partial^2_{x_i}+\partial_t,
\end{align} and $C$ is a positive constant that depends on $M,
N,\overline{T},\varepsilon$.
\end{theorem}
\noindent Theorem \ref{th:error_estimates_taylor} will be proved in Section \ref{smalla}.
\begin{remark}\label{rem:put}
Theorem \ref{th:error_estimates_taylor} can be extended by relaxing the regularity hypotheses on
the terminal data $\phi$. More precisely, if $k\in\N$, it is sufficient to assume that $\phi \in
C^{k-1}$ and the that derivatives are \emph{locally} Lipschitz continuous with exponential growth
at infinity.  In this case, estimate \eqref{eq:error_estimate} would be modified by substituting
the constant $C$ by $C e^{C|x|}$. As we shall see in Section \ref{sec:model}, such an extension
would allow for including some important functions $\phi$ commonly used in financial applications,
such as the Call payoff function.  Even though this generalization does not change the core of the
proof of Theorem \ref{th:error_estimates_taylor}, in order to avoid an excess of technicalities,
we shall continue our analysis under the more restrictive hypotheses of Theorem
\ref{th:error_estimates_taylor}.
\end{remark}

We remark explicitly that \eqref{eq:error_estimate} does not imply convergence as $N$ goes to
infinity because the constant $C$, appearing in the estimate, depends on $N$ and, in principle,
this constant can blow up in the limit {as $N \to \infty$}. Thus, the usefulness of Theorem
\ref{th:error_estimates_taylor} is as an asymptotic estimate for small times.
\par
Now, we state more general convergence estimates that are valid on any time interval $[t,T]$. For
any $m\in\mathbb{N}$ we consider the equispaced partition $\{t_{0},\dots,t_{m}\}$ of $[t,T]$
defined as
\begin{align}
t_k
    &:= t + k \, \delta_{m} , &
\delta_{m}
    &:= \frac{T-t}{m} .
\end{align}
Moreover, we set
\begin{align}\label{def:u_N_m}
\bar{u}_{N,m}(t_0,x_0)
    &:= \int\limits_{\mathbb{R}^{m d}} \prod_{i=1}^{m}
            \bar{\Gamma}_{N}(t_{i-1},x_{i-1};t_{i},x_{i})\phi(x_m)\, \dd x_1 \cdots \dd x_m, &
x_0
    &\in\mathbb{R}^d,
\end{align}
where $\bar{\Gamma}_{N}$ is the $N$th order Taylor approximation of $\G$.
\begin{theorem}\label{th:bootstrapping_convergence}
Assume $N\ge 1$. Under the assumptions of Theorem \ref{th:error_estimates_taylor} we have
\begin{align}\label{eq:convergence}
 \left|u(t,x)-\bar{u}_{N,m}(t,x) \right| \leq C \left(\frac{T-t}{m}\right)^{\frac{N+k-1}{2}}, \qquad 0\leq t<T\leq \overline{T},\ x\in\mathbb{R}^d,
\end{align}
where the constant $C$ only depends on $M,N,\overline{T}$ and
$\left\|\varphi\right\|_{C^{k-1,1}_{b}}$.
\end{theorem}
\noindent Theorem \ref{th:bootstrapping_convergence} will be proved in Section \ref{boota}. We note explicitly that, as a direct consequence of \eqref{eq:convergence}, we have
that if $N\geq 2-k$ then
\begin{align}
 \lim_{m\to\infty}\bar{u}_{N,m}(t,x)= u(t,x), \qquad t\in[0,T],\ x\in\mathbb{R}^d.
\end{align}
\begin{remark}
\label{rmk:bootstrap}
From \eqref{eq:un} and \eqref{eq:new} we see that
\begin{align}
\bar{\Gam}_N(t,x;T,y)
    &=  \( 1  + \sum_{i=1}^N  \Lc_i^x(t,T) \) \Gam_0(t,x;T,y) .
\end{align}
When the differential operator $(1 + \sum_i \Lc_i^x)$ hits the Gaussian kernel $\Gam_0(t,x;T,y)$
it simply returns a polynomial of $(x,y)$ times the Gaussian kernel $\Gam_0(t,x;T,y)$.  The
coefficients $(a_{\alpha,0})_{|\alpha|\leq 2}$ of the operator $(1 + \sum_i \Lc_i^x)$ also depend
on $x$ and are smooth by Assumption \ref{assumption:parabol_holder_bonded}, condition part ii).
Thus, evaluating \eqref{def:u_N_m} involves computing an $(d \cdot m)$-dimensional integral,
where the integrand is the product of Gaussian kernels with polynomials and smooth, bounded
coefficients.  Since the integrand is smooth and slowly varying, these integrals can be computed
numerically without major difficulties.  Though, clearly, there is a limit to how large $(d \cdot
m)$ can be.
\end{remark}

\section{Applications to financial mathematics}
\setcounter{equation}{0}
\label{sec:model} In this section we motivate our analysis by
illustrating how our methodology applies to the pricing derivatives in financial mathematics.  To
begin, we consider an arbitrage-free market. We take, as given, an equivalent martingale measure
$\Qb$ defined on a complete filtered probability space $(\Om,\Fc,\{\Fc_t,t\geq0\})$. All
stochastic processes defined below live on this probability space and all expectations are taken
with respect to $\Qb$. The risk-neutral dynamics of our market are described by the following
$d$-dimensional Markov diffusion
\begin{align}
\dd X_t
   &=  \mu(t,X_t) \dd t + \sig(t,X_t) \dd W_t . \label{eq:StochVol}
\end{align}
Here $W$ is a standard $m$-dimensional Brownian motion, the function $\mu : \Rb_+ \times \Rb^d \to
\Rb^d$ and the function $\sig: \Rb_+ \times \Rb^d \to \Rb^{ d \times m}$.  The components of $X$
could represent a number of things, e.g., economic factors, asset prices, economic indicators, or functions of these quantities.
{In particular, we assume a risk-free interest rate of the form $r(t,X_t)$ where $r : \Rb_+ \times \Rb^d \to \Rb_+$.}
We also introduce a random time $\zeta$, which is given by
\begin{align}
\zeta
    &=  \inf \big\{ t \geq 0 : \int_0^t \gam(s,X_s) \dd s \geq \Ec \big\} , &
\gam
        &:  { \Rb_+ \times \Rb^d \to \Rb_+ } ,
\end{align}
with $\Ec$ exponentially distributed and independent of $X$. The random time $\zeta$ could
represent the default time of an asset, the arrival of an economic shock, etc..
\par
Denote by $V$ the no-arbitrage price of a European derivative expiring at time $T$ with payoff
\begin{align}
H(X_T) \, \Ib_{\{\zeta > T\}} + G(X_T) \, \Ib_{\{\zeta \leq T\}}
    &=  \big( H(X_T) - G(X_T) \big) \, \Ib_{\{\zeta > T\}} + G(X_T) .
\end{align}
It is well known (see, for instance, \cite{yorbook}) that
\begin{align}\label{e1}
V_t
    &= { \Eb \[ \ee^{-\int_t^T r(s,X_s)  \dd s} G(X_T) | X_t \] +
        \Ib_{\{\zeta>t\}} \Eb \[\ee^{-\int_t^T \( r(s,X_s) + \gam(s,X_s) \) \dd s} \Big(H(X_T)-G(X_T) \Big)  | X_t \] , } &
t
    &<T .
\end{align}
Then, to value a European-style option, one must compute functions of the form
\begin{align}
u(t,x)
    &:= \Eb \[\ee^{- \int_t^T { \lam(s,X_s) } \dd s}  \phi(X_T) \mid X_t = x \] . \label{eq:v}
\end{align}
Under mild assumptions, the function $u$, defined by \eqref{eq:v}, satisfies the Kolmogorov
backward equation
\begin{align} \label{Cauchy_prob}
\begin{cases}
 (\d_t + \Ac) u(t,x)=  0 ,\qquad & t<T,\ x \in \Rb^d,\\
 u(T,x)=  \phi(x), &x \in \Rb^d,
\end{cases}
\end{align}
where the operator $\Ac$ is given explicitly by
\begin{align}
 \Ac &=  \frac{1}{2} \sum_{i,j=1}^d  \left(\sig \sig^\text{T}\right)_{ij}(t,x) \d_{x_i}\d_{x_j}
                + \sum_{i=1}^d \mu_i(t,x) \d_{x_i} - { \lam(t,x) } . \label{eq:A}
\end{align}
The results of Section \ref{sec:approximating1} give an explicit and effective way to construct
closed-form approximate solutions of problem \eqref{Cauchy_prob}, and therefore closed-form
approximate option prices \eqref{eq:v}. The rigorous error bounds prove the efficiency of the
approach and confirm the high accuracy of the approximation in financial applications.  {For the interested reader, extensive numerical examples can be found in \cite{pagliarani2011analytical}, \cite{pascucci}, \cite{lorig-pagliarani-pascucci-1}, \cite{lorig-pagliarani-pascucci-2} and \cite{lorig-pagliarani-pascucci-3}.}

\section{Proof of Theorem \ref{th:un_general_repres}: analytical approximation formulas}\label{appendix:proof}\setcounter{equation}{0}
The proof is based on the symmetry properties of the Gaussian fundamental solution
$\Gamma_{0}=\Gamma_{0}(t,x;s,\xi)$ as it is defined in \eqref{e22and}-\eqref{eq:covariance_mean},
combined with an extensive use of other very general relations such as the Duhamel's principle and
the Chapman-Kolmogorov equation which we recall for completeness.
\begin{lemma}[Chapman-Kolmogorov identity]\label{lem:semigroup_property}
Under Assumption \ref{assumption:parabol_holder_bonded}, for any $t<s<T$, $x,y\in\mathbb{R}^d$, we
have
\begin{align}\label{eq:semigroup_property}
\int_{\mathbb{R}^d} \Gamma(t,x;s,\xi)\Gamma(s,\xi;T,y)\, \dd \xi = \Gamma(t,x;T,y).
\end{align}
\end{lemma}

We start by recalling the operator
\begin{align}
 \Mc^{x}(t,s) = x+\mv(t,s)+ \Cv(t,s)\nabla_x
\end{align}
as it is defined in \eqref{eq:M}.  Above, and throughout the proof, we use the superscript $x$ to
explicitly indicate the variables on which the operator acts. Furthermore, we define the operator
\begin{align}\label{def_Mbar}
 \bar{\Mc}^{y}(t,s) = y- \mv(t,s)+ \Cv(t,s)\nabla_y.
\end{align}
{The following lemma illustrates how the operator $\nabla_x$ relates to $\nabla_y$ when acting on
$\Gamma_{0}(t,x;s,y)$ and how the multiplication operators $y$ and $x$ relate to $\Mc^{x}(t,s)$
and $\bar{\Mc}^{y}(t,s)$ respectively, when acting on $\Gamma_{0}(t,x;s,y)$.}
\begin{lemma}
For any $t<s$ and $x,y\in\mathbb{R}^d$, we have
\begin{align}
\nabla_x \Gamma_{0}(t,x;s,y)
    &= -\nabla_{y} \Gamma_{0}(t,x;s,y),
\label{prop:deriv}
\end{align}
and
\begin{align}\label{prop:polyn1}
 y\, \Gamma_{0}(t,x;s,y) =&  \Mc^{x}(t,s)\Gamma_{0}(t,x;s,y),\\ \label{prop:polyn2} x\, \Gamma_{0}(t,x;s,y) =&
\bar{\Mc}^{y}(t,s)\Gamma_{0}(t,x;s,y).
\end{align}
\end{lemma}
\begin{proof}
While the previous identities can be directly verified  a posteriori by elementary computations,
here we give an alternative ``constructive'' proof which shows {how to find $\Mc^x$-like and
$\bar{\Mc}^y$-like operators, which are equivalent to multiplication by the backward and forward
variables $y$ and $x$ respectively, in even more general frameworks (see Remark \ref{rand1}
below).} To this end, we will require some properties of the Fourier transform
\begin{align}
\Fc_{x}f(\xi)
&:=\frac{1}{\sqrt{(2\pi)^{d}}}\int_{\mathbb{R}^{d}}\ee^{\ii x\xi}f(x) \dd x.
\end{align}
First, we recall that for any function $f$ in the Schwartz space we have
\begin{align}\label{prop:fourier}
\ii \xi \Fc_{x}(f)&=\Fc_{x}(-\nabla_{x}f) , &
 \Fc_{x}(x f)&=-\ii\nabla_{\xi}\Fc_{x}f .
\end{align}
Moreover, we have
\begin{align}\label{e22}
 \Fc_{x}\Gamma_{0}(t,\cdot;T,y)(\xi)
    &=  \frac{1}{  \sqrt{(2\pi)^{d}} }
    e^{\ii \xi(y -\mv(t,T)) -\frac{1}{2}\langle\Cv(t,T)\xi,\xi\rangle},\\ \label{e22bis}
 \Fc_{y}\Gamma_{0}(t,x;T,\cdot)(\eta)
    &=  \frac{1}{  \sqrt{(2\pi)^{d}} }
    e^{\ii \eta(x+\mv(t,T)) -\frac{1}{2}\langle\Cv(t,T)\eta,\eta\rangle}.
\end{align}
To obtain the identity \eqref{prop:deriv} we simply use that $\Gam_0(t,x;T,y)=\Gam_0(t,x-y;T,0)$.
For \eqref{prop:polyn1}, we have:
\begin{align}
 \Fc_{y}(y \Gamma_{0})&=-\ii\nabla_{\eta}\Fc_{y}(\Gamma_{0}) \\
    &= \left(x+\mv(t,s)+ \Cv(t,s)\ii \eta\right)
    \Fc_{y}(\Gamma_{0}) &
    & \text{(by \eqref{e22bis})} \\
  &=  \Fc_{y}(\left(x+\mv(t,s)- \Cv(t,s)\nabla_{y}\right)\Gamma_{0}) &
    &\text{(by \eqref{prop:fourier})}\\
  &=  \Fc_{y}(\Mc^{x}(t,s)\Gamma_{0}). &
    &\text{(by \eqref{prop:deriv})}
  \end{align}
The proof of identity \eqref{prop:polyn2} is analogous to the proof of identity \eqref{prop:polyn1}.
\end{proof}

\begin{remark}\label{rand1}
It is worth noting that the argument of the above proof applies whenever {the characteristic
function of the stochastic process with transition density $\Gam_0$ is explicitly known and when
$\Gam_0$ can be expressed as a function of $x-y$.   Thus, $\Mc^x$-like and $\bar{\Mc}^y$-like
operators can be obtained, for example, when $\Gam_0$ is the transition density of an additive
(i.e., time-dependent L\'evy) process.  In this case, the $\Mc^x$-like and $\bar{\Mc}^y$-like
operators would be pseudo-differential operator rather than (usual) differential operators.}
\end{remark}

\begin{corollary}\label{corollary_welldef}
For any $t<s$, $s_{1}\in[0,T]$ and $x,y \in\mathbb{R}^d$ we have
\begin{align}\label{prop:polin_xi}
 a_{\alpha,n}(s_{1},y) \Gamma_{0}(t,x;s,y)= &\, a_{\alpha,n}\left(s_{1},\Mc^x(t,s)\right) \Gamma_{0}(t,x;s,y) ,\\
 \label{prop:polin_x}
 a_{\alpha,n}(s_{1},x) \Gamma_{0}(t,x;s,y)= &\,
 a_{\alpha,n}\left(s_{1},\bar{\Mc}^{y}(t,s)\right) \Gamma_{0}(t,x;s,y) .
\end{align}
\end{corollary}
\begin{proof} First we note that the components $\Mc_i^{x}(t,s)$, $i=1,\dots,d$,
of the operator $\Mc^{x}(t,s)$ commute when applied to $\Gamma_{0}=\Gamma_{0}(t,x;s,y)$ and to its
derivatives (notice however that this is not true in general when they are applied to a generic
function).  Indeed, for any multi-index $\b$, we have
\begin{align}
  \Mc_i^{x}(t,s)\Mc_j^{x}(t,s)D^{\b}_{x}\Gamma_{0}&=(-1)^{|\b|}\Mc_i^{x}(t,s)\Mc_j^{x}(t,s)D^{\b}_{y}\Gamma_{0}
  &  &\text{(by \eqref{prop:deriv})}\\
  &=(-1)^{|\b|}D^{\b}_{y}\Mc_i^{x}(t,s)\Mc_j^{x}(t,s)\Gamma_{0}\\
  &=(-1)^{|\b|}D^{\b}_{y}\Mc_i^{x}(t,s)y_{j}\Gamma_{0} &  &\text{(by \eqref{prop:polyn1})}\\
  &=(-1)^{|\b|}D^{\b}_{y}y_{j}\Mc_i^{x}(t,s)\Gamma_{0}\\
  &=(-1)^{|\b|}D^{\b}_{y}y_{j}y_{i}\Gamma_{0} \\
  &=\Mc_j^{x}(t,s)\Mc_i^{x}(t,s)D^{\b}_{x}\Gamma_{0}. & &\text{{(by reversing the steps above)}}
\end{align}
Since $a_{\alpha,n}(s_{1},\cdot)$ is a polynomial by construction, we therefore have that the
operators $a_{\alpha,n}\left(s_{1},\Mc^x(t,s)\right)$
are defined unambiguously when applied to $\Gamma_{0}(t,x;s,y)$ and to its derivatives. Moreover,
clearly \eqref{prop:polin_xi} is now a straightforward consequence of \eqref{prop:polyn1}. An
analogous argument shows the validity of \eqref{prop:polin_x}.
\end{proof}

We now recall the operators
\begin{align}
  \Ac^x_n(s)
        &= \sum_{|\alpha |\leq 2}  a_{\alpha,n}(s,x) D_x^{\alpha}, &
\Gc^x_{n}(t,s)
        &= \sum_{|\alpha| \leq 2} a_{\alpha,n}\big(s,\Mc^x(t,s)\big) D_x^{\alpha}, &
n
        &\geq 0 , \label{eq:G.def}
\end{align}
as they are defined in \eqref{eq:A.expand} and \eqref{def_Gn}, and we introduce the operator
\begin{align}
{\bar{\Gc}^y_{n}(t,s)}
    &= \sum_{|\alpha| \leq 2}(-1)^{|\alpha |} D_y^{\alpha} a_{\alpha,n}\big(t,{\bar{\Mc}^y(t,s)}\big), &
n
    &\geq 0, \label{eq:Gbar.def}
\end{align}
with ${\bar{\Mc}^y}$ as in \eqref{def_Mbar}. We remark explicitly that, by Corollary
\ref{corollary_welldef}, operators $\Gc^x_{n}(t,s)$ and ${\bar{\Gc}^y_{n}(t,s)}$ are defined
unambiguously when applied to $\Gamma_{0}=\Gamma_{0}(t,x;s,y)$, to its derivatives and, more
generally, {by the representation formula \eqref{e10}}, to solutions of the Cauchy problem
\eqref{eq:v.0.pide}.

The next proposition and its remarkable corollaries are the key of the proof of Theorem
\ref{th:un_general_repres}.
\begin{proposition}
For any $t<s<T$, $x,y\in\mathbb{R}^d$ and $n\geq 1$, we have
\begin{align}\label{eq:fund_prop}
\int_{\mathbb{R}^{d}} \Gamma_{0}(t,x;s,\xi) \Ac^{\xi}_n(s) f(\xi)  \dd \xi
&=\Gc^{x}_{n}(t,s)\int_{\mathbb{R}^{d}} \Gamma_{0}(t,x;s,\xi) f(\xi)  \dd \xi ,\\
\label{eq:fund_prop1} \int_{\mathbb{R}^{d}} f(\xi) \Ac^{\xi}_n(s) \Gamma_{0}(s,\xi;T,y) \dd \xi
&=\bar{\Gc}^{y}_{n}(s,T)\int_{\mathbb{R}^{d}} f(\xi) \Gamma_{0}(s,\xi;T,y)  \dd \xi ,\end{align}
for any $f\in C^2_{0}\left( \mathbb{R}^d \right)$. Furthermore, the following relation holds:
\begin{align}\label{eq:fund_prop2}
\Gc^{x}_{n}(t,s)\Gamma_{0}(t,x;T,y)=\bar{\Gc}^{y}_{n}(s,T)\Gamma_{0}(t,x;T,y).
\end{align}
\end{proposition}
\begin{proof}
We first prove \eqref{eq:fund_prop}. By the definition of $\Ac^{\xi}_n$ we have
\begin{align}
&\int_{\mathbb{R}^d} \Gamma_{0}(t,x;s,\xi)\, \Ac^{\xi}_n(s) f(\xi)  \dd \xi = \sum_{|\alpha| \leq
2} \int_{\mathbb{R}^d} a_{\alpha,n}(s,\xi)  \Gamma_{0}(t,x;s,\xi) \,D_{\xi}^{\alpha} f(\xi) \dd
\xi\\
&= \sum_{|\alpha| \leq 2}  a_{\alpha,n}\big(s,\Mc^x(t,s)\big) \int_{\mathbb{R}^d}
\Gamma_{0}(t,x,y;s,\xi,\omega) \,D_{\xi}^{\alpha} f(\xi)  \dd \xi & &\text{(by
\eqref{prop:polin_xi})}
\\ &= \sum_{|\alpha| \leq 2}  a_{\alpha,n}\big(s,\Mc^x(t,s)\big) (-1)^{|\alpha |}
\int_{\mathbb{R}^d} D_{\xi}^{\alpha} \Gamma_{0}(t,x;s,\xi) \, f(\xi)  \dd \xi &
&\text{(integrating by parts)} \\ &= \sum_{|\alpha| \leq 2}  a_{\alpha,n}\big(s,\Mc^x(t,s)\big)
D_{x}^{\alpha} \int_{\mathbb{R}^d}  \Gamma_{0}(t,x;s,\xi) \, f(\xi)  \dd \xi & &\text{(by
\eqref{prop:deriv})} \\ &= \Gc^{x}_{n}(t,s) \int_{\mathbb{R}^d}  \Gamma_{0}(t,x;s,\xi) \, f(\xi)
\dd \xi . & &\text{{(by \eqref{eq:G.def}, definition of $\Gc^{x}_{n}$)}}
\end{align}
Similarly, for \eqref{eq:fund_prop1}, {using the definition of $\Ac^{\xi}_n$ we have}
\begin{align}
&\int_{\mathbb{R}^d} f(\xi)\, \Ac^{\xi}_n(s) \Gamma_{0}(s,\xi;T,y)  \dd \xi  = \sum_{|\alpha| \leq
2} \int_{\mathbb{R}^d} f(\xi) \, a_{\alpha,n}(s,\xi) \, D_{\xi}^{\alpha} \Gamma_{0}(s,\xi;T,y) \dd
\xi
\\
&= \sum_{|\alpha| \leq 2} (-1)^{|\alpha |} D_{y}^{\alpha} \int_{\mathbb{R}^d} f(\xi) \,
a_{\alpha,n}(s,\xi) \, \Gamma_{0}(s,\xi;T,y)  \dd \xi & &\text{(by \eqref{prop:deriv})} \\ & =
\sum_{|\alpha| \leq 2} (-1)^{|\alpha |} D_{y}^{\alpha}\,
a_{\alpha,n}\big(s,\bar{\Mc}^{y}(s,T)\big) \int_{\mathbb{R}^d} f(\xi) \, \Gamma_{0}(s,\xi;T,y) \dd
\xi & &\text{(by \eqref{prop:polin_x})} \\ &=\bar{\Gc}^{y}_{n}(s,T) \int_{\mathbb{R}^d} f(\xi) \,
\Gamma_{0}(s,\xi;T,y)  \dd \xi . & &\text{{(by \eqref{eq:Gbar.def}, definition of
$\bar{\Gc}^{y}_{n}$)}}
\end{align}
Identity \eqref{eq:fund_prop2} follows from \eqref{eq:fund_prop} and \eqref{eq:fund_prop1}.
Indeed, using the Chapman-Kolmogorov equation we have
\begin{align}
&\Gc^{x}_{n}(t,s)\Gamma_{0}(t,x;T,y) =  \Gc^{x}_{n}(t,s) \int_{\mathbb{R}^d} \Gamma_{0}(t,x;s,\xi)
\, \Gamma_{0}(s,\xi;T,y)  \dd \xi
\\
& = \int_{\mathbb{R}^d} \Gamma_{0}(t,x;s,\xi) \, \Ac^{\xi}_n(s) \Gamma_{0}(s,\xi;T,y)  \dd \xi &
&\hspace{-2em}\text{(applying \eqref{eq:fund_prop} with {$f(\xi)=\Gamma_{0}(s,\xi;T,y)$})} \\
&=\bar{\Gc}^{y}_{n}(s,T) \int_{\mathbb{R}^d} \Gamma_{0}(t,x;s,\xi) \, \Gamma_{0}(s,\xi;T,y) \dd
\xi  & &\hspace{-2em}\text{(applying \eqref{eq:fund_prop1} with {$f(\xi)=\Gamma_{0}(t,x;s,\xi)$})} \\ &=
\bar{\Gc}^{y}_{n}(s,T) \Gamma_{0}(t,x;T,y). & &\hspace{-2em}\text{{(by Chapman-Kolmogorov)}}
\end{align}
\end{proof}
\begin{corollary}\label{fund_corol}
For any $t<s<T$, $x,y\in\mathbb{R}$, $n\geq 1$, we have
\begin{align}\label{fund_coroll_1}
\int_{\mathbb{R}^d}  \Gamma_{0}(t,x;s,\xi)\,  \Gc^{\xi}_{i_1}(s,s_1)\cdots \Gc^{\xi}_{i_n}(s,s_n)
\Gamma_{0}(s,\xi;T,y)\dd \xi  =&\, \Gc^{x}_{i_1}(t,s_1)\cdots \Gc^{x}_{i_n}(t,s_n)
\Gamma_{0}(t,x;T,y),
\end{align}
for any $i\in\mathbb{N}^n$ and $s<s_1<\cdots <s_n<T$.
\end{corollary}
\begin{proof}
We first prove \eqref{fund_coroll_1}. By induction on $n$. For $n=1$, and for any $i_1\geq 1$,
$t<s_1<T$, we have
\begin{align}
&\int_{\mathbb{R}^d}  \Gamma_{0}(t,x;s,\xi)\,  \Gc^{\xi}_{i_1}(s,s_1) \Gamma_{0}(s,\xi;T,y)\dd \xi
\\ &=\, \bar{\Gc}^{y}_{i_1}(s_1,T) \int_{\mathbb{R}^d}  \Gamma_{0}(t,x;s,\xi)\, \Gamma_{0}(s,\xi;T,y)\dd
\xi & &\text{(by \eqref{eq:fund_prop2})} \\ &=\, \bar{\Gc}^{y}_{i_1}(s_1,T) \Gamma_{0}(t,x;T,y) &
&\text{(by Chapman-Kolmogorov)} \\ &=\, \Gc^{x}_{i_1}(t,s_1) \Gamma_{0}(t,x;T,y). & &\text{(by
\eqref{eq:fund_prop2})}
\end{align}
We assume now the thesis to be true for $n\geq 1$ and for any $i\in\mathbb{N}^n$, $s<s_1,\cdots
s_n<T$. Then, for any $i_{n+1}\geq 1$, $s_n<s_{n+1}<T$ we have
\begin{align}
&\int_{\mathbb{R}^d}  \Gamma_{0}(t,x;s,\xi)\,  \Gc^{\xi}_{i_1}(s,s_1)\cdots \Gc^{\xi}_{i_n}(s,s_n)
\Gc^{\xi}_{i_{n+1}}(s,s_{n+1}) \Gamma_{0}(s,\xi;T,y)\dd \xi  \\ &=
\bar{\Gc}^{y}_{i_{n+1}}(s_{n+1},T) \int_{\mathbb{R}^d}  \Gamma_{0}(t,x;s,\xi)\,
\Gc^{\xi}_{i_1}(s,s_1)\cdots \Gc^{\xi}_{i_n}(s,s_n) \Gamma_{0}(s,\xi;T,y)\dd \xi  &
&\text{\big(\eqref{eq:fund_prop2} on $\Gc^{\xi}_{i_{n+1}}(s,s_{n+1}) \Gamma_{0}$\big)} \\ &=
\bar{\Gc}^{y}_{i_{n+1}}(s_{n+1},T)\, \Gc^{x}_{i_1}(t,s_1)\cdots \Gc^{x}_{i_n}(t,s_n)
\Gamma_{0}(t,x;T,y) & &\text{{(inductive hypothesis)}}\\ &= \Gc^{x}_{i_1}(t,s_1)\cdots
\Gc^{x}_{i_n}(t,s_n)\, \bar{\Gc}^{y}_{i_{n+1}}(s_{n+1},T) \Gamma_{0}(t,x;T,y)  \\ &=
\Gc^{x}_{i_1}(t,s_1)\cdots \Gc^{x}_{i_n}(t,s_n)\, \Gc^{x}_{i_{n+1}}(t,s_{n+1})
\Gamma_{0}(t,x;T,y), & &\text{(\eqref{eq:fund_prop2} on $\bar{\Gc}^{y}_{i_{n+1}}(s_{n+1},T)
\Gamma_{0}$)}
\end{align}
which proves \eqref{fund_coroll_1}.
\end{proof}
\noindent From here to the end of this {section}, we set $\gam=0$.  We do this merely to save
space. The general case, with $\gam \neq 0$, is completely analogous and introduces no
complications.
\begin{corollary}\label{fund_corol2}
Let $u_{0}$ be as in \eqref{e10} with $\gamma=0$. For any $t<s<T$, $x,y\in\mathbb{R}$, $n\geq 1$,
we have
\begin{align}\label{fund_coroll_2}
 \int_{\mathbb{R}^d}  \Gamma_{0}(t,x;s,\xi)\,  \Gc^{\xi}_{i_1}(s,s_1)\cdots
 \Gc^{\xi}_{i_n}(s,s_n)  u_0(s,\xi)\dd \xi  =&\, \Gc^{x}_{i_1}(t,s_1)\cdots \Gc^{x}_{i_n}(t,s_n) u_0(t,x),
\end{align}
for any $i\in\mathbb{N}^n$ and $s<s_1<\cdots <s_n<T$.
\end{corollary}
\begin{proof}
By \eqref{e10} we have
\begin{align}
&\int_{\mathbb{R}^d} \Gamma_{0}(t,x;s,\xi)\, \Gc^{\xi}_{i_{1}}(s,s_1) \cdots
\Gc^{\xi}_{i_{n}}(s,s_n)u_0(s,\xi) \dd\xi \\ &=  \int_{\mathbb{R}^d} \Gamma_{0}(t,x;s,\xi)\,
\Gc^{\xi}_{i_{1}}(s,s_1) \cdots \Gc^{\xi}_{i_{n}}(s,s_n) \,
        \int_{\mathbb{R}^d}  \Gamma_{0}(s,\xi;T,y)\, \phi(y) \dd y \, { \dd\xi  } \\
&=
\int_{\mathbb{R}^d} \phi(y)  \int_{\mathbb{R}^d} \Gamma_{0}(t,x;s,\xi) \, \Gc^{\xi}_{i_{1}}(s,s_1)
\cdots \Gc^{\xi}_{i_{n}}(s,s_n) \Gamma_{0}(s,\xi;T,y)
\dd\xi \, {\dd y} & &\text{(Fubini's theorem)} \\ &=
\int_{\mathbb{R}^d} \phi(y)\,  \Gc^{x}_{i_{1}}(t,s_1) \cdots \Gc^{x}_{i_{n}}(t,s_n)
\Gamma_{0}(t,x;T,y) \dd y & &\text{(by Corollary \ref{fund_corol})} \\ &= \Gc^{x}_{i_{1}}(t,s_1)
\cdots \Gc^{x}_{i_{n}}(t,s_n) u_0(t,x), & &\text{(by \eqref{e10})}
\end{align}
which concludes the proof.
\end{proof}

We are now in position to prove Theorem \ref{th:un_general_repres}.
Proceeding by induction on $n$, we first prove the case $n=1$. By definition, $u_1$ is the unique
solution of the non-homogeneous Cauchy problem \eqref{eq:v.n.pide} with $n=1$. Thus, by
Duhamel's principle we have
\begin{align}
&u_1(t,x)=\int_t^T \int_{\mathbb{R}^d} \Gamma_{0}(t,x;s,\xi)\, \Ac_1^{\xi}(s) u_0(s,\xi) \dd\xi
\dd s
\\ &=\int_t^T \Gc^{x}_{1}(t,s) \int_{\mathbb{R}^d} \Gamma_{0}(t,x;s,\xi)\, u_0(s,\xi) \dd \xi \,
{\dd s }& &\text{(by \eqref{eq:fund_prop} with $n=1$)} \\ &=\int_t^T \Gc^{x}_{1}(t,s)
\int_{\mathbb{R}^d} \Gamma_{0}(t,x;s,\xi) \int_{\mathbb{R}^d} \Gamma_{0}(s,\xi;T,y) \phi(y) \dd y
\, {\dd \xi \, \dd s} & &\text{(by \eqref{e10})} \\ &=  \int_t^T \Gc^{x}_{1}(t,s)
\int_{\mathbb{R}^d} \phi(y) \int_{\mathbb{R}^d} \Gamma_{0}(t,x;s,\xi)\, \Gamma_{0}(s,\xi;T,y) \dd
\xi \, { \dd y \, \dd s } & &\text{(Fubini's theorem)} \\ &=\int_t^T \Gc^{x}_{1}(t,s) \dd s\,
u_0(t,x) & &\text{(Chapman-Kolmogorov and \eqref{e10})}\\ &=\Lc^{x}_1(t,T) u_0(t,x). & &\text{(by
\eqref{eq:def_Ln}-\eqref{eq:def_Ln_bis}) }
\end{align}
For the general case, let us assume that \eqref{eq:un} holds for $n\geq 1$, and prove it holds for
$n+1$. By definition, $u_{n+1}$ is the unique solution of the non-homogeneous Cauchy problem
\eqref{eq:v.n.pide}. Thus, by Duhamel's principle, we have
\begin{align}
&u_{n+1}(t,x)=\int_t^T \int_{\mathbb{R}^d} \Gamma_{0}(t,x;s,\xi) \sum^{n+1}_{h=1} \Ac_h^{\xi}(s)
u_{n+1-h}(s,\xi) \dd\xi \dd s \\ &=  \sum^{n+1}_{h=1} \int_t^T \Gc^{x}_{h}(t,s)
\int_{\mathbb{R}^d} \Gamma_{0}(t,x;s,\xi)\, u_{n+1-h}(s,\xi) \dd\xi \, { \dd s} & &\text{(by
\eqref{eq:fund_prop} with $n=h$)} \\ &=  \sum^{n+1}_{h=1} \int_t^T
\Gc^{x}_{h}(t,s)\int_{\mathbb{R}^d} \Gamma_{0}(t,x;s,\xi)\, \Lc^{\xi}_{n+1-h}(s,T) u_0(s,\xi)
\dd\xi \, { \dd s}.\label{eq:u_n+1} & &\text{(by induction hypothesis)}
\end{align}
Now, by definition \eqref{eq:def_Ln}-\eqref{eq:def_Ln_bis} we have
\begin{align}
&\int_{\mathbb{R}^d} \Gamma_{0}(t,x;s,\xi)\, \Lc^{\xi}_{n+1-h}(s,T) u_0(s,\xi) \dd\xi  \\
&=\sum_{j=1}^{n+1-h}  \int_{\mathbb{R}^d} \Gamma_{0}(t,x;s,\xi)  \int_s^T \dd s_1 \cdots
\int_{s_{j-1}}^T \dd s_j \sum_{i\in I_{n+1-h,j}}\Gc^{\xi}_{i_{1}}(s,s_1) \cdots
\Gc^{\xi}_{i_{j}}(s,s_j) u_0(s,\xi) \, { \dd\xi }\\ &=\sum_{j=1}^{n+1-h}  \int_s^T \dd s_1
\cdots \int_{s_{j-1}}^T \dd s_j \sum_{i\in I_{n+1-h,j}} \int_{\mathbb{R}^d}
\Gamma_{0}(t,x;s,\xi)\, \Gc^{\xi}_{i_{1}}(s,s_1) \cdots \Gc^{\xi}_{i_{j}}(s,s_j) u_0(s,\xi) \dd\xi
& &\text{(Fubini's theorem)} \\[-2em] \label{eq:sum_G} &=\sum_{j=1}^{n+1-h} \int_s^T \dd s_1
\cdots \int_{s_{j-1}}^T \dd s_j  \sum_{i\in I_{n+1-h,j}}  \Gc^{x}_{i_{1}}(t,s_1) \cdots
\Gc^{x}_{i_{j}}(t,s_j) u_0(t,x). & &\text{(by \eqref{fund_coroll_2})}
\end{align}
Next, by inserting \eqref{eq:sum_G} into \eqref{eq:u_n+1} we obtain
\begin{align}
u_{n+1}(t,x)= \tilde{\Lc}^{x}_{n}(t,T) u_0(t,x),
\end{align}
where
\begin{align}
\tilde{\Lc}^{x}_{n}(t,T)
    &=  \int_{t}^T \Gc^{x}_{n+1}(t,s_{0}) \dd s_{0} \\ & \qquad
            + \sum^{n}_{h=1} \, \sum_{j=1}^{n+1-h} \int_{t}^T \dd s_{0}
\int_{s_{0}}^T \dd s_1 \cdots \int_{s_{j-1}}^T \dd s_j  \hspace{-5pt}\sum_{i\in I_{n+1-h,j}}\hspace{-10pt}
\Gc^{x}_{h}(t,s_{0}) \Gc^{x}_{i_{1}}(t,s_1) \cdots \Gc^{x}_{i_{j}}(t,s_j) .
\end{align}
In order to conclude the proof, it is enough to check that
$\tilde{\Lc}^{x}_{n}(t,T)=\Lc^{x}_{n+1}(t,T)$. By exchanging the indexes in the sums, we {obtain}
\begin{align}
 \tilde{\Lc}^{x}_{n}(t,T)
    &=\int_{t}^T \Gc^{x}_{n+1}(t,s_{0}) \dd s_{0} \\ & \qquad
        +\sum^{n}_{j=1} \, \sum_{h=1}^{n+1-j}
    \int_{t}^T \dd s_{0} \int_{s_{0}}^T \dd s_1 \cdots \int_{s_{j-1}}^T \dd s_j \hspace{-5pt} \sum_{i\in I_{n+1-h,j}}\hspace{-10pt}
    \Gc^{x}_{h}(t,s_{0})  \Gc^{x}_{i_{1}}(t,s_1) \cdots \Gc^{x}_{i_{j}}(t,s_j)
\intertext{(setting $l=j+1$)}
 &=\int_{t}^T \Gc^{x}_{n+1}(t,s_{0}) \dd s_{0}\\
 &\quad+\sum^{n+1}_{l=2} \, \sum_{h=1}^{n+2-l}
 \int_{t}^T \dd s_{0} \int_{s_{0}}^T \dd s_1 \cdots \int_{s_{l-2}}^T \dd s_{l-1} \hspace{-5pt} \sum_{i\in I_{n+1-h,l-1}}\hspace{-10pt}
 \Gc^{x}_{h}(t,s_{0})\Gc^{x}_{i_{1}}(t,s_1) \cdots \Gc^{x}_{i_{l-1}}(t,s_{l-1})
\intertext{(replacing the integration variables: $(\dd s_{0}, \dd s_1,\cdots, \dd s_{l-1})\to (\dd
r_1, \dd r_2,\cdots, \dd r_{l})$)}
 &=\int_{t}^T \Gc^{x}_{n+1}(t,s_{0}) \dd s_{0} \\ & \qquad
        +\sum^{n+1}_{l=2} \, \sum_{h=1}^{n+2-l}
 \int_{t}^T \dd r_1 \int_{r_1}^T \dd r_2 \cdots \int_{r_{l-1}}^T \dd r_{l} \hspace{-5pt} \sum_{i\in I_{n+1-h,l-1}}\hspace{-10pt}
 \Gc^{x}_{h}(t,r_1)\Gc^{x}_{i_{1}}(t,r_2) \cdots \Gc^{x}_{i_{l-1}}(t,r_{l})\\
 &=\int_{t}^T \Gc^{x}_{n+1}(t,s_{0}) \dd s_{0} \\ & \qquad
        +\sum^{n+1}_{l=2} \int_{t}^T \dd r_1
 \int_{r_1}^T \dd r_2 \cdots \int_{r_{l-1}}^T \dd r_{l} \sum_{h=1}^{n+2-l}
 \sum_{i\in I_{n+1-h,l-1}} \hspace{-10pt} \Gc^{x}_{h}(t,r_1) \Gc^{x}_{i_{1}}(t,r_2) \cdots
 \Gc^{x}_{i_{l-1}}(t,r_{l})
\intertext{(by definition \eqref{eq:def_Ln_bis})}
 &=\sum^{n+1}_{l=1}
 \int_{t}^T \dd r_1 \int_{r_1}^T \dd r_2 \cdots\int_{r_{l-1}}^T \dd r_{l}
 \sum_{z\in I_{n+1,l}} \Gc^{x}_{z_1}(t,r_1) \Gc^{x}_{z_{2}}(t,r_2)
 \cdots \Gc^{x}_{z_{l}}(t,r_{l})
\intertext{(by definition \eqref{eq:def_Ln})}
 &=\Lc^{x}_{n+1}(t,T),
\end{align}
which concludes the proof. $\hfill \Box$

\section{Proof of Theorem \ref{th:error_estimates_taylor}:  error bounds for small times}\label{smalla}\setcounter{equation}{0}
Throughout this section we fix $M$, $N$ and $\overline{T}$.  All of the constants appearing in the
estimates proved in this section \emph{depend on $M,N$ and $\overline{T}$} and will not continue
repeating this below. Under the main Assumption \ref{assumption:parabol_holder_bonded}, {the
operator $(\d_t + \Ac)$} admits a unique fundamental solution $\Gamma=\Gamma(t,x;T,y)$ for which
the following classical Gaussian estimates hold (see \cite{friedman-parabolic}, Chapter 1).
\begin{lemma}\label{lem:gaussian_estimates}
For any $\varepsilon>0$ and $\b,\ggg\in\mathbb{N}_{0}^{d}$ with $|\ggg|\le N+2$, we have
\begin{align}\label{Gaua}
 |(x-y)^{\beta}\, D_x^{\ggg}\Gamma(t,x;T,y)|\le C \cdot (T-t)^{\frac{|\b|-|\ggg|}{2}}\Gamma^{M+\varepsilon}(t,x;T,y),\qquad 0\leq t<T \leq\overline{T},\quad x,y\in\R^{d},
\end{align}
where $\Gamma^{M+\varepsilon}$ is the fundamental solution of the heat operator
\eqref{eq:heatoperator} and $C$ is a positive constant, only dependent on
$M,N,\overline{T},\varepsilon$ and $|\b|$.
\end{lemma}
In order to state our theoretical results we need some preliminary estimates on the spatial
derivatives of the solution of the Cauchy problem with {coefficients that may depend on $t$ but
are constant in $x$}. The quality of such estimates depends on the regularity of the terminal data
$\varphi$.
\begin{proposition}\label{l1}
Assume the coefficients of $\Ac$ to be constant in space (i.e. $a_{\alpha}(t,\cdot)\equiv
a_{\alpha}(t)$). Let $\b\in\N_{0}^{d}$ and $\phi\in C_{b}^{k-1,1}\left(\mathbb{R}^d\right)$ for
some $k\in\N_{0}$. Then the solution of the Cauchy problem \eqref{eqpde} satisfies
\begin{align}\label{e1bb}
  \left|D_{x}^{\beta}u(t,x)\right|\le C \cdot (T-t)^{\frac{\min\{k-|\b|,0\}}{2}},\qquad 0\leq t<T\leq\overline{T},\ x\in\mathbb{R}^d,
\end{align}
where $C$ only depends on $M,N,\overline{T},|\b|$ and $\left\|\varphi\right\|_{C^{k-1,1}_{b}}$.
\end{proposition}
\begin{proof}
As $\Ac$ has space-independent coefficients, the fundamental solution of $(\d_t+\Ac)$ 
is the Gaussian function in \eqref{e22and}. A direct computation shows that for any polynomial
function $p=p(y)$ we have
\begin{align}
\int\limits_{\mathbb{R}^d}  p(y)\Gamma_{0}(t,x;T,y)\, \dd y= \bar{p}(x),
\end{align}
where $\bar{p}$ is a polynomial with degree $\mathrm{deg}(\bar{p})=\mathrm{deg}(p)$. Thus, for any
$\ggg\in\N_{0}^{d}$ with $|\ggg|> \mathrm{deg}(p)$ we have
\begin{align}\label{eq:integral_polynom_derivat}
 \int\limits_{\mathbb{R}^d}p(y) D_x^{\ggg} \Gamma_{0}(t,x;T,y) \, \dd y=D_x^{\ggg} \int\limits_{\mathbb{R}^d}  p(y) \Gamma_{0}(t,x;T,y)\, \dd y = 0.
\end{align}
In particular, let us set $h=\min\{|\b|,k\}$ and denote by $ \mathbf{T}^{\varphi}_{\xbar,h}$ the
$h$-th order Taylor polynomial of $\varphi$ centered at $\xbar$, i.e.,
\begin{align}\label{eqtayl}
 \mathbf{T}^{\varphi}_{\xbar,h}(x)=\sum_{|\ggg|\leq
 h}\frac{D^{\ggg}\varphi(\xbar)}{\ggg!}(x-\xbar)^{\ggg} ,
\end{align}
where, by convention, when $h=-1$, then $\mathbf{T}^{\varphi}_{\xbar,-1}\equiv 0$. Then we have
\begin{align}\label{eq:integral_Taylor_derivat}
 \int\limits_{\mathbb{R}^d}\mathbf{T}^{\varphi}_{x,h-1}(y) D_x^{\beta} \Gamma_{0}(t,x;T,y) \, \dd y= 0.
\end{align}
Now, by Duhamel's principle we have
\begin{align}\label{e10-2}
 u(t,x)= \ee^{\int_t^T \g(s) \dd s} \int\limits_{\mathbb{R}^d} \Gamma_{0}(t,x;T,y) \varphi(y)\, \dd y ,\quad t<T,\ x\in\mathbb{R}^d .
\end{align}
Next, since $\varphi\in C_{b}^{k-1,1}(\mathbb{R}^d)$, by \eqref{eq:integral_Taylor_derivat} we
obtain
\begin{align}
D_x^{\beta}  u(t,x)= \ee^{\int_t^T \g(s) \dd s} \int\limits_{\mathbb{R}^d} \varphi(y) D_x^{\beta}
\Gamma_{0}(t,x;T,y)\, \dd y =\ee^{\int_t^T \g(s) \dd s}\int\limits_{\mathbb{R}^d} (\varphi(y)-
\mathbf{T}^{\varphi}_{x,h-1}(y)) D_x^{\beta} \Gamma_{0}(t,x;T,y)\, \dd y.
\end{align}
Thus, {by the Taylor theorem with integral remainder,} we obtain
\begin{align}
 \left|D_x^{\beta}  u(t,x)\right| \leq C \int\limits_{\mathbb{R}^d} |x-y|^{h} \left|D_x^{\beta} \Gamma_{0}(t,x;T,y)\right|\, \dd y,
\end{align}
where $C$ depends on $\left\| \varphi \right\|_{C_{b}^{k-1,1}}$. The thesis follows from Lemma
\ref{lem:gaussian_estimates} and from
\begin{align}
\int\limits_{\mathbb{R}^d} \Gamma^{M+\varepsilon}(t,x;T,y)\, \dd y=1.
\end{align}
\end{proof}
\noindent Hereafter, we assume all the hypotheses of Theorem \ref{th:error_estimates_taylor} are
satisfied.  The proof of Theorem \ref{th:error_estimates_taylor} is based on the following lemmas.
\begin{lemma}\label{lemma:taylor1}
Under the hypotheses of Theorem \ref{th:error_estimates_taylor}, for any $\bar{x}\in\mathbb{R}^d$
and $N\in\N_{0}$, we have
\begin{align}
 u(t,x)-\bar{u}_{N}^{(\bar{x})}(t,x) = \int_t^T \int\limits_{\mathbb{R}^d} \Gamma(t,x;s,\xi) \sum_{n=0}^N
 \left(\Ac - \bar{\Ac}^{(\bar{x})}_n\right) u^{(\bar{x})}_{N-n}(s,\xi)\, \dd \xi \dd s, \quad
 t<T,\ x\in\mathbb{R}^d,
\end{align}
where {the function $u$ is the solution of \eqref{eqpde}, the function} $\bar{u}_{N}^{(\bar{x})}$ is
the $N$th order approximation in \eqref{Taylor_approx_k_N_xbar} and
\begin{align}
\bar{\Ac}^{(\bar{x})}_n = \sum_{i=0}^n  \Ac^{(\bar{x})}_i.
\end{align}
\end{lemma}
\begin{proof}
We first prove the identity
\begin{align}\label{key_lemma_eq}
 \left(\partial_t +\Ac\right) \bar{u}_{N}^{(\bar{x})} (t,x) =  \sum_{n=0}^N (\Ac -
 \bar{\Ac}^{(\bar{x})}_n) u^{(\bar{x})}_{N-n}(t,x), \quad t<T,\ x\in\mathbb{R}^d.
\end{align}
For $N=0$ we have
\begin{align}
\left(\partial_t +\Ac\right) \bar{u}_{0}^{(\bar{x})}
 &= \left(\Ac - \Ac^{(\bar{x})}_0\right) u^{(\bar{x})}_{0} ,
\end{align}
because $\left(\partial_t + \Ac^{(\bar{x})}_0\right) u^{(\bar{x})}_{0}= 0 $ by definition
\eqref{eq:v.0.pide.xbar}. We assume now \eqref{key_lemma_eq} holds for $N\geq 0$ and we prove it
to hold for $N+1$.  We have
\begin{align}
&\left(\partial_t +\Ac\right) \bar{u}_{N+1}^{(\bar{x})}  \\
    &= (\partial_t +\Ac) \bar{u}_{N}^{(\bar{x})}  +(\partial_t
            +\Ac) u^{(\bar{x})}_{N+1} \\
    &= \sum_{n=0}^N \big(\Ac - \bar{\Ac}^{(\bar{x})}_n\big) u^{(\bar{x})}_{N-n}  +\big(\Ac-\Ac_0^{(\bar{x})}\big)            u^{(\bar{x})}_{N+1}  -
            \sum\limits_{n=1}^{N+1} \Ac^{(\bar{x})}_n u^{(\bar{x})}_{N+1-n} &
    &\hspace{-3em}\text{(by inductive hypothesis and by \eqref{eq:v.n.pide.xbar})} \\
    &= \sum_{n=1}^{N+1} \left(\Ac - \bar{\Ac}^{(\bar{x})}_{n-1}\right) u^{(\bar{x})}_{N+1-n}
 +\big(\Ac-\Ac_0^{(\bar{x})}\big) u^{(\bar{x})}_{N+1}  - \sum\limits_{n=1}^{N+1} \Ac^{(\bar{x})}_n u^{(\bar{x})}_{N+1-n} &
    &\hspace{-3em}\text{(by shifting the index of  the first sum)}\\
    &= \sum_{n=1}^{N+1} \big(\Ac - \bar{\Ac}^{(\bar{x})}_{n}\big) u^{(\bar{x})}_{N+1-n}  +\big(\Ac-\Ac_0^{(\bar{x})}\big)
 u^{(\bar{x})}_{N+1}=
 \sum_{n=0}^{N+1} \big(\Ac - \bar{\Ac}^{(\bar{x})}_{n}\big) u^{(\bar{x})}_{N+1-n}.
\end{align}
Now, since $u$ is the classical solution of \eqref{eqpde}, we have by \eqref{key_lemma_eq}
that $v:=u-\bar{u}_{N}^{(\bar{x})}$ solves the following problem
\begin{align}
\begin{cases}
 (\d_t + \Ac ) v(t,x) =- \sum\limits_{n=0}^N (\Ac - \bar{\Ac}^{(\bar{x})}_n) u^{(\bar{x})}_{N-n}(t,x),
 \qquad & t<T,\ x\in\mathbb{R}^d, \\
 v(T,x) =  0, &  x \in\mathbb{R}^d,
\end{cases}
\end{align}
The thesis follows by Duhamel's principle.
\end{proof}
\begin{lemma}\label{lem:der_u0}
Under the assumptions of Theorem \ref{th:error_estimates_taylor}, for any multi-index $\b
\in\mathbb{N}_0^{d}$ we have
\begin{align}\label{e1b}
  \left|D_{x}^{\b}u^{(\bar{x})}_0(t,x)\right|&\le C \cdot (T-t)^{\frac{\min\{k-|\b|,0\}}{2}},\qquad \quad 0\leq t<T\leq \overline{T},\quad
  x,\bar{x}\in\mathbb{R}^d.
\intertext{Moreover, if $N\ge 1$ then for any $n\in\N$, $n\leq N$, we have} \label{eq:der_u_n}
 \left|D^{\beta}_{x}u^{(\bar{x})}_n(t,x)\right|&\le C \cdot (T-t)^{\frac{n+k-|\b|}{2}}\left(1+|x-\bar{x}|^{n}\left(T-t\right)^{-\frac{n}{2}}
 \right),\qquad 0\leq t<T\leq \overline{T},\ x,\bar{x}\in\mathbb{R}^d.
\end{align}
The constants in \eqref{e1b} and \eqref{eq:der_u_n} depend only on $M,N,\overline{T},|\b|$ and
$\left\|\phi\right\|_{C^{k-1,1}_{b}}$.
\end{lemma}
\begin{proof}
In this proof, $\{C_{i}\}_{i\ge1}$ denote some positive constants that depend only on
$M,N,\overline{T}$ and $\left\|\phi\right\|_{C^{k-1,1}_{b}}$.
For clarity, write the operators appearing in Theorem \ref{th:un_general_repres} as ${\Lc}^{x,(\bar{x})}_k$ and ${\Gc}^{x,(\bar{x})}_k$ in order to indicate that these operators are constructed using the expansion point $\xb$ and act on the variable $x$.
\par
For $n=0$, the thesis follows
directly from Proposition \ref{l1} since $u^{(\bar{x})}_0$ solves problem \eqref{eq:v.0.pide.xbar}.
Next we prove the assertion for $n=1$.  By Theorem \ref{th:un_general_repres}, for any
$\bar{x}\in\mathbb{R}^d$ we have
\begin{align}
 u_1^{(\bar{x})}(t,x)&= \Lc^{x,(\bar{x})}_1(t,T) u^{(\bar{x})}_0(t,x)
 =\int_t^T \Gc^{x,(\bar{x})}_{1}(t,s) u^{(\bar{x})}_0(t,x)\, \dd s
 \\
 &= \sum_{|\ggg|\leq 2}\int_t^T a^{(\bar{x})}_{\ggg,1}\left(s,x+\mv^{(\bar{x})}(t,s)+ \Cv^{(\bar{x})}(t,s)\nabla_{x}\right)\, \dd
 s\,  D^{\ggg}_{x}  u^{(\bar{x})}_0(t,x)\\
\intertext{(by \eqref{eq:A.expand.xbar} with $n=1$)}\label{and1}
 &=\sum_{|\ggg|\leq 2} \int_t^T \langle\nabla_{x} a_{\ggg}(s,\bar{x}),x-\bar{x}+\mv^{(\bar{x})}(t,s)+ \Cv^{(\bar{x})}(t,s)\nabla_{x}\rangle\, \dd s\,  D^{\ggg}_{x}
 u^{(\bar{x})}_0(t,x).
\end{align}
Therefore we obtain
  $$\left|D^{\beta}_{x}u^{(\bar{x})}_1(t,x)\right|\le I_{1}+I_{2}+I_{3}+I_{4}$$
where
\begin{align}
I_1
    &=  \sum_{|\ggg|\leq 2} \int_t^T \left|\nabla_{x} a_{\ggg}(s,\bar{x})\right| \dd s\,  \left|x-\bar{x}\right|\left|D^{\b+\ggg}_{x}
  u^{(\bar{x})}_0(t,x)\right|, \\
I_2
    &=  \sum_{|\ggg|\leq 2} \int_t^T \left|\nabla_{x} a_{\ggg}(s,\bar{x})\right|\left|\mv^{(\bar{x})}(t,s)\right| \dd s\, \left|D^{\b+\ggg}_{x}
  u^{(\bar{x})}_0(t,x)\right|, \\
I_3
    &=  \sum_{|\ggg|\leq 2} \int_t^T \left|\nabla_{x} a_{\ggg}(s,\bar{x})\right|\left|\Cv^{(\bar{x})}(t,s)\right| \dd s\, \left|\nabla_{x}D^{\b+\ggg}_{x}
  u^{(\bar{x})}_0(t,x)\right|, \\
I_4
    &=  \sum_{|\ggg|\leq 2\atop |\delta|\leq |\b|-1} \int_t^T \left|\nabla_{x} a_{\ggg}(s,\bar{x})\right| \dd s\,  \left|D^{\ggg+\delta}_{x}
  u^{(\bar{x})}_0(t,x)\right|.
\end{align}
Now, since $a_{\ggg}\in C^{1,1}_{b}$, by Proposition \ref{l1} we have
\begin{align}
  I_{1}
 &\le C_{1}\sum_{|\ggg|\leq 2}\left|x-\bar{x}\right| (T-t)^{\frac{2+\min\{k-|\b|-|\ggg|,0\}}{2}}\le C_{2}  \cdot  (T-t)^{\frac{1+k-|\b|}{2}} \frac{\left|x-\bar{x}\right|}{\sqrt{T-t}}.
\end{align}
Moreover, since $a_{\ggg}\in C^{1,1}_{b}$ and
$\left|\mv^{(\bar{x})}(t,s)\right|\le C_{3}(s-t)$, we have by Proposition \ref{l1} that
\begin{align}
  I_{2}
 &\le C_{4}\sum_{|\ggg|\leq 2} (T-t)^{2+\frac{\min\{k-|\b|-|\ggg|,0\}}{2}}\le C_{5}  \cdot (T-t)^{\frac{2+k-|\b|}{2}}.
\end{align}
Next, since $a_{\ggg}\in C^{1,1}_{b}$ and
$\left|\Cv^{(\bar{x})}(t,s)\right|\le C_{6} \cdot (s-t)$, we have by Proposition \ref{l1} that
\begin{align}
  I_{3}
 &\le C_{7}\sum_{|\ggg|\leq 2} (T-t)^{2+\frac{\min\{k-1-|\b|-|\ggg|,0\}}{2}}\le C_{8}  \cdot (T-t)^{\frac{1+k-|\b|}{2}}.
\end{align}
Finally, we have the term appearing when $D^{\b}_{x}$ applies to $x-\xbar$ in \eqref{and1}. Using
the same arguments as above we obtain
\begin{align}
   I_{4}\le C_{9}\sum_{|\ggg|\leq 2} (T-t)^{\frac{2+\min\{k+1-|\b|-|\ggg|,0\}}{2}}\le C_{10}  \cdot (T-t)^{\frac{1+k-|\b|}{2}}.
\end{align}
Using all the above estimates, one deduces \eqref{eq:der_u_n} for $n=1$. The general case can be
proved by analogous arguments, using repeatedly the general expression of $u_n^{(\bar{x})}$ provided
by Theorem \ref{th:un_general_repres} and the estimates of Proposition \ref{l1}. We omit the
details for brevity.
\end{proof}
\noindent We are now in the position to prove Theorem \ref{th:error_estimates_taylor}.
\begin{proof}[Proof of Theorem \ref{th:error_estimates_taylor}]
In this proof, $\{C_{i}\}_{i\ge1}$ denote some positive constants dependent only on
$M,N,\overline{T}$ and $\left\|\phi\right\|_{C^{k-1,1}_{b}}$. By Lemma \ref{lemma:taylor1} we have
\begin{align}
 u-\bar{u}_{N}=\sum_{n=0}^N I_n,\qquad\qquad
 I_n(t,x)=\int_t^T \int\limits_{\mathbb{R}^d} \Gamma(t,x;s,\xi)  \left( \Ac - \sum_{i=0}^n \Ac^{x}_i  \right) u^{x}_{N-n}(s,\xi)\, \dd \xi \dd
 s.
\end{align}
Moreover $I_{n}=I_{n,1}+I_{n,2}$ with (cf. \eqref{eqtayl})
\begin{align}
  I_{n,1}(t,x)&=\sum_{|\alpha|\leq 1}\int_t^T \int\limits_{\mathbb{R}^d}
  \left( a_{\alpha}(s,\xi) -
  \mathbf{T}^{a_{\alpha}(s,\cdot)}_{x,n}(\xi)
  \right)\Gamma(t,x;s,\xi)D_{\xi}^{\alpha}
  u^{x}_{N-n}(s,\xi)\, \dd \xi \dd s,\\
  I_{n,2}(t,x)&=\sum_{|\alpha|=2}\int_t^T \int\limits_{\mathbb{R}^d}
  \left( a_{\alpha}(s,\xi) -
  \mathbf{T}^{a_{\alpha}(s,\cdot)}_{x,n}(\xi)
  \right)\Gamma(t,x;s,\xi) D_{\xi}^{\alpha}
  u^{x}_{N-n}(s,\xi)\, \dd \xi \dd s.
\end{align}
Now by Lemma \ref{lem:der_u0} we have
\begin{align}
 \left|I_{n,1}(t,x)\right|&\leq C_{1} \sum_{|\alpha|\leq 1}\int_t^T \int\limits_{\mathbb{R}^d}|\xi-x|^{n+1}\Gamma(t,x;s,\xi)(T-s)^{\frac{N-n-|\alpha|+k}{2}}
  \left(1+ (T-s)^{-\frac{N-n}{2}}|x-\xi|^{N-n}  \right)\, \dd \xi \dd s \\
 &\leq C_{2} \sum_{|\alpha|\leq 1}\int_t^T\left((T-s)^{\frac{N-n+|\a|+k}{2}}(s-t)^{\frac{n+1}{2}}
 +(T-s)^{\frac{-|\a|+k}{2}}(s-t)^{\frac{N+1}{2}}\right)\int\limits_{\mathbb{R}^d}\Gamma^{M+\eps}(t,x;s,\xi)\, \dd \xi \dd
 s\\
 &\le C_{3} \cdot (T-t)^{\frac{N+k+2}{2}}
\end{align}
where we have used Lemma \ref{lem:gaussian_estimates} and the identity
  $$\int_{t}^{T}(T-s)^{n}(s-t)^{k}\, \dd s=\frac{\Gamma_E (k+1) \Gamma_E (n+1)}{\Gamma_E (k+n+2)}(T-t)^{k+n+1},$$
{with $\Gamma_E$ denoting the Euler Gamma function.} To estimate $I_{n,2}$ we first integrate by
parts and obtain
 $$ I_{n,2}(t,x)=-\sum_{|\alpha_1| =1}\sum_{|\alpha_2| = 1}
  \int_t^T \int\limits_{\mathbb{R}^d}
  D_{\xi}^{\alpha_{1}}\left(\left( a_{\alpha_{1}+\alpha_{2}}(t,\xi) -
  \mathbf{T}^{a_{\alpha_{1}+\alpha_{2}}(t,\cdot)}_{x,n}(\xi)
  \right)\Gamma(t,x;s,\xi)\right)D_{\xi}^{\alpha_{2}}
  u^{x}_{N-n}(s,\xi)\, \dd \xi \dd s.$$
Using the same arguments as above one can show that $$\left|I_{n,2}(t,x)\right|\leq C_{4} \cdot
(T-t)^{\frac{N+k+1}{2}}.$$ Finally estimate \eqref{th:error_estim_fund_solution} is obtained by a
straightforward modification of the proof of \eqref{eq:error_estimate} for $k=0$, by means of the
application of Lemma \ref{lem:gaussian_estimates} and the Chapman-Kolmogorov equation. We omit the
details for simplicity.
\end{proof}
\section{Proof of Theorem \ref{th:bootstrapping_convergence}: error bounds for large times}\label{boota}\setcounter{equation}{0}
In agreement with the hypothesis of Theorem \ref{th:bootstrapping_convergence}, throughout this
section we will assume $N\geq 1$.  The proof of Theorem \ref{th:bootstrapping_convergence} is
based on the Chapman-Kolmogorov identity \eqref{eq:semigroup_property} and on the following
classical Schauder estimate (see, for instance, \cite{friedman-parabolic}, Chapter 3).
\begin{lemma}\label{lem:der_u}
Let $u$ be the solution of problem \eqref{eqpde} under Assumption
\ref{assumption:parabol_holder_bonded}.
Then for $0\leq k\leq 2$, we
have
  $$\left\|u(t,\cdot)\right\|_{C^{k-1,1}_{b}(\R^{d})}\le C \left\|\varphi\right\|_{C^{k-1,1}_{b}(\R^{d})},
  \qquad 0\le t\le T\le \overline{T}, $$
where $C$ is a positive constant that depends only on $M$ and $\overline{T}$.
\end{lemma}
\begin{proof}[Proof of Theorem \ref{th:bootstrapping_convergence}] In this proof, $\{C_{i}\}_{i\ge1}$ denote some positive constants that depend only on $M,N,\overline{T}$ and
$\left\|\phi\right\|_{C^{k-1,1}_{b}}$.  By an iterative use of \eqref{eq:semigroup_property}, the Chapman-Kolmogorov identity, we have
\begin{align}
 u(t_0,x_0)=\int\limits_{\mathbb{R}^{m d}}
 \prod_{i=1}^{m} \Gamma(t_{i-1},x_{i-1};t_{i},x_{i})\phi(x_m)\, \dd x_m \cdots \dd x_1, \quad t_0<T,\quad x_0\in\mathbb{R}^d.
\end{align}
Then, by definition \eqref{def:u_N_m} we obtain
\begin{align}\label{eq:sum_Ii}
 u-\bar{u}_{N,m}=\sum_{j=1}^m I_j,
\end{align}
where
\begin{align}
I_j(t_0,x_0)
&= \int\limits_{\mathbb{R}^{m d}} \prod_{i=1}^{j-1}
 \bar{\Gamma}_{N}(t_{i-1},x_{i-1};t_{i},x_{i})\left(\bar{\Gamma}_{N}-\Gamma\right)(t_{j-1},x_{j-1};t_{j},x_{j}) \\ & \qquad
 \times \prod_{i=j+1}^{m} \Gamma(t_{i-1},x_{i-1};t_{i},x_{i})\, \phi(x_m)\, \dd x_m \cdots \dd x_1 \\
&=\int\limits_{\mathbb{R}^{(j-1) d}}\,\prod_{i=1}^{j-1}
 \bar{\Gamma}_{N}(t_{i-1},x_{i-1};t_{i},x_{i})\int\limits_{\mathbb{R}^{d}}
 \left(\bar{\Gamma}_{N}-\Gamma\right)(t_{j-1},x_{j-1};t_{j},x_{j})
 u(t_j,x_j) \dd x_j \, \dd x_{j-1} \cdots \dd x_{1},
\end{align}
where we have used Fubini's theorem and the Chapman-Kolmogorov identity.
Now by Lemma \ref{lem:der_u} and Theorem \ref{th:error_estimates_taylor} we obtain
\begin{align}
 \left|  \int_{\mathbb{R}^{d}}\left(\bar{\Gamma}_{N} - \Gamma\right)(t_{j-1},x_{j-1};t_{j},x_{j})\,
 u(t_j,x_j) \, \dd x_j \right|\leq
 C_1\delta_{m}^{\frac{N+k+1}{2}}.
\end{align}
Thus, we have
\begin{align}
 |I_j(t_0,x_0)|&\leq C_1 \delta_{m}^{\frac{N+k+1}{2}} \int\limits_{\mathbb{R}^{(j-1)d}}
 \prod_{i=1}^{j-1} |\bar{\Gamma}_{N}(t_{i-1},x_{i-1};t_{i},x_{i})|\, \dd x_{j-1} \cdots \dd
 x_{1}\\
 &\leq C_1 \delta_{m}^{\frac{N+k+1}{2}} \int\limits_{\mathbb{R}^{(j-1)d}}
 \prod_{i=1}^{j-1} \left(\left|\bar{\Gamma}_{N}-\Gamma\right|+\Gamma\right)(t_{i-1},x_{i-1};t_{i},x_{i})\, \dd x_{j-1} \cdots \dd
 x_{1} \\
 &\leq C_1 \delta_{m}^{\frac{N+k+1}{2}} \int\limits_{\mathbb{R}^{(j-1)d}}
 \prod_{i=1}^{j-1} \left(C_2\delta_{m}^{\frac{N+1}{2}}\Gamma^{M+1}+\Gamma\right)(t_{i-1},x_{i-1};t_{i},x_{i})\, \dd x_{j-1} \cdots \dd
 x_{1},
\end{align}
where, in the last step we used Eq. \eqref{th:error_estim_fund_solution} in Theorem \ref{th:error_estimates_taylor}, 
with $\Gamma^{M+1}$ being the fundamental solution of the heat-type operator
\eqref{eq:heatoperator} with $\eps=1$. Therefore, by applying repeatedly the properties
$$
\int_{\mathbb{R}^d}\Gamma^{M+1}(t,x;s,y) \dd y =1,\qquad  \int_{\mathbb{R}^d}\Gamma(t,x;s,y) \dd y \leq 1,
$$
we obtain
\begin{align}
| I_j(t_0,x_0)|&\leq  C_1 \delta_{m}^{\frac{N+k+1}{2}}\left(  C_2 \delta_{m}^{\frac{N+1}{2}} + 1
\right)^{j-1}.
\end{align}
Eventually, since $N\ge 1$, we find by \eqref{eq:sum_Ii} that
\begin{align}
 \left| u(t,x)-\bar{u}_{N,m}(t,x) \right| \leq C_1 \left( C_2
 \delta_{m}^{\frac{N+1}{2}} + 1  \right)^{m}m\, \delta_{m}^{\frac{N+k+1}{2}} \leq C_3 e^{C_2(T-t)^{\frac{N+1}{2}}}\delta_{m}^{\frac{N+k-1}{2}},
\end{align}
which proves \eqref{eq:convergence}.
\end{proof}

\subsection*{Acknowledgments}
The authors would like the thank two anonymous referees, whose suggestions improved both the readability and mathematical quality of this manuscript.

%
%

\bibliographystyle{chicago}
\bibliography{Bibtex-Master-3.00}

\begin{thebibliography}{}

\bibitem[\protect\citeauthoryear{Avramidi}{Avramidi}{2007}]{avramidi2007analytic}
Avramidi, I.~G. (2007).
\newblock Analytic and geometric methods for heat kernel applications in
  finance.
\newblock {\em New Mexico Institute of Mining and Technology, NM\/}~{\em
  87801}.

\bibitem[\protect\citeauthoryear{Benhamou, Gobet, and Miri}{Benhamou
  et~al.}{2010}]{BenhamouGobetMiri2010}
Benhamou, E., E.~Gobet, and M.~Miri (2010).
\newblock Expansion formulas for {E}uropean options in a local volatility
  model.
\newblock {\em Int. J. Theor. Appl. Finance\/}~{\em 13\/}(4), 603--634.

\bibitem[\protect\citeauthoryear{Berline, Getzler, and Vergne}{Berline
  et~al.}{1992}]{berline1992heat}
Berline, N., E.~Getzler, and M.~Vergne (1992).
\newblock {\em Heat kernels and Dirac operators}.
\newblock Springer.

\bibitem[\protect\citeauthoryear{Cheng, Costanzino, Liechty, Mazzucato, and
  Nistor}{Cheng et~al.}{2011}]{ChengCostanzinoLiechtyMazzucatoNistor2011}
Cheng, W., N.~Costanzino, J.~Liechty, A.~Mazzucato, and V.~Nistor (2011).
\newblock Closed-form asymptotics and numerical approximations of 1{D}
  parabolic equations with applications to option pricing.
\newblock {\em SIAM J. Financial Math.\/}~{\em 2\/}(1), 901--934.

\bibitem[\protect\citeauthoryear{Corielli, Foschi, and Pascucci}{Corielli
  et~al.}{2010}]{pascucci-parametrix}
Corielli, F., P.~Foschi, and A.~Pascucci (2010).
\newblock Parametrix approximation of diffusion transition densities.
\newblock {\em SIAM J. Financial Math.\/}~{\em 1}, 833--867.

\bibitem[\protect\citeauthoryear{Fouque, Papanicolaou, Sircar, and
  Solna}{Fouque et~al.}{2011}]{fpss}
Fouque, J.-P., G.~Papanicolaou, R.~Sircar, and K.~Solna (2011).
\newblock {\em Multiscale stochastic volatility for equity, interest rate, and
  credit derivatives}.
\newblock Cambridge: Cambridge University Press.

\bibitem[\protect\citeauthoryear{Friedman}{Friedman}{1964}]{friedman-parabolic}
Friedman, A. (1964).
\newblock {\em Partial differential equations of parabolic type}.
\newblock Englewood Cliffs, N.J.: Prentice-Hall Inc.

\bibitem[\protect\citeauthoryear{Friz, Gerhold, and Yor}{Friz
  et~al.}{2013}]{frizyor2013}
Friz, P.~K., S.~Gerhold, and M.~Yor (2013).
\newblock How to make {D}upire's local volatility work with jumps.
\newblock {\em arXiv preprint1302.5548\/}.

\bibitem[\protect\citeauthoryear{Goldstein}{Goldstein}{1980}]{goldstein}
Goldstein, H. (1980).
\newblock {\em Classical Mechanics\/} (2nd ed.).
\newblock Reading, MA: Addison-Wesley Publishing Company.

\bibitem[\protect\citeauthoryear{Hagan and Woodward}{Hagan and
  Woodward}{1999}]{hagan-woodward}
Hagan, P. and D.~Woodward (1999).
\newblock Equivalent black volatilities.
\newblock {\em Applied Mathematical Finance\/}~{\em 6\/}(3), 147--157.

\bibitem[\protect\citeauthoryear{Henry-Labord{\`e}re}{Henry-Labord{\`e}re}{2009}]{laborderebook}
Henry-Labord{\`e}re, P. (2009).
\newblock {\em Analysis, geometry, and modeling in finance: Advanced methods in
  option pricing}, Volume~13.
\newblock Chapman \& Hall.

\bibitem[\protect\citeauthoryear{Heston}{Heston}{1993}]{heston1993}
Heston, S. (1993).
\newblock {A closed-form solution for options with stochastic volatility with
  applications to bond and currency options}.
\newblock {\em Rev. Financ. Stud.\/}~{\em 6\/}(2), 327--343.

\bibitem[\protect\citeauthoryear{{Jacquier} and {Lorig}}{{Jacquier} and
  {Lorig}}{2013}]{lorig-jacquier-1}
{Jacquier}, A. and M.~{Lorig} (2013).
\newblock The smile of certain l{\'e}vy-type models.
\newblock {\em ArXiv preprint arXiv:1207.1630\/}.

\bibitem[\protect\citeauthoryear{Jeanblanc, Yor, and Chesney}{Jeanblanc
  et~al.}{2009}]{yorbook}
Jeanblanc, M., M.~Yor, and M.~Chesney (2009).
\newblock {\em Mathematical methods for financial markets}.
\newblock Springer Verlag.

\bibitem[\protect\citeauthoryear{Kevorkian and Cole}{Kevorkian and
  Cole}{1996}]{KevorkianCole1996}
Kevorkian, J. and J.~D. Cole (1996).
\newblock {\em Multiple scale and singular perturbation methods}, Volume 114 of
  {\em Applied Mathematical Sciences}.
\newblock New York: Springer-Verlag.

\bibitem[\protect\citeauthoryear{Lagerstrom}{Lagerstrom}{1988}]{Lagerstrom1988}
Lagerstrom, P.~A. (1988).
\newblock {\em Matched asymptotic expansions}, Volume~76 of {\em Applied
  Mathematical Sciences}.
\newblock New York: Springer-Verlag.
\newblock Ideas and techniques.

\bibitem[\protect\citeauthoryear{Lorig}{Lorig}{2013}]{lorig-3}
Lorig, M. (2013).
\newblock The exact smile of certain local volatility models.
\newblock {\em Quantitative Finance\/}~{\em 13\/}(6), 897--905.

\bibitem[\protect\citeauthoryear{Lorig, Pagliarani, and Pascucci}{Lorig
  et~al.}{2013a}]{lorig-pagliarani-pascucci-1}
Lorig, M., S.~Pagliarani, and A.~Pascucci (2013a).
\newblock A family of density expansions for {L}{\'e}vy-type processes with
  default.
\newblock {\em to appear in Annals of Applied Probability\/}.

\bibitem[\protect\citeauthoryear{Lorig, Pagliarani, and Pascucci}{Lorig
  et~al.}{2013b}]{lorig-pagliarani-pascucci-2}
Lorig, M., S.~Pagliarani, and A.~Pascucci (2013b).
\newblock Implied vol for any local-stochastic vol model.
\newblock {\em ArXiv preprint arXiv:1306.5447\/}.

\bibitem[\protect\citeauthoryear{Lorig, Pagliarani, and Pascucci}{Lorig
  et~al.}{2013c}]{lorig-pagliarani-pascucci-3}
Lorig, M., S.~Pagliarani, and A.~Pascucci (2013c).
\newblock A {T}aylor series approach to pricing and implied vol for {LSV}
  models.
\newblock {\em ArXiv preprint arXiv:1308.5019\/}.

\bibitem[\protect\citeauthoryear{Murray}{Murray}{2002}]{Murray2002}
Murray, J.~D. (2002).
\newblock {\em Mathematical biology. {I}\/} (Third ed.), Volume~17 of {\em
  Interdisciplinary Applied Mathematics}.
\newblock New York: Springer-Verlag.
\newblock An introduction.

\bibitem[\protect\citeauthoryear{Pagliarani and Pascucci}{Pagliarani and
  Pascucci}{2012}]{pagliarani2011analytical}
Pagliarani, S. and A.~Pascucci (2012).
\newblock Analytical approximation of the transition density in a local
  volatility model.
\newblock {\em Cent. Eur. J. Math.\/}~{\em 10\/}(1), 250--270.

\bibitem[\protect\citeauthoryear{Pagliarani, Pascucci, and Riga}{Pagliarani
  et~al.}{2013}]{pascucci}
Pagliarani, S., A.~Pascucci, and C.~Riga (2013).
\newblock Adjoint expansions in local {L}\'evy models.
\newblock {\em SIAM J. Financial Math.\/}~{\em 4}, 265--296.

\bibitem[\protect\citeauthoryear{Sakurai}{Sakurai}{1994}]{sakurai}
Sakurai, J.~J. (1994).
\newblock {\em Modern Quantum Mechanics\/} ({R}evised ed.).
\newblock Reading, MA: Addison-Wesley Publishing Company.

\bibitem[\protect\citeauthoryear{Van~Dyke}{Van~Dyke}{1975}]{VanDyke1975}
Van~Dyke, M. (1975).
\newblock {\em Perturbation methods in fluid mechanics\/} (Annotated ed.).
\newblock The Parabolic Press, Stanford, Calif.

\end{thebibliography}

\end{document}